\documentclass[a4paper,12pt,reqno]{article}
 \usepackage[english]{babel}
  \usepackage{amsmath,amsfonts,amssymb,amsthm,bbm}
   \usepackage{graphics,epsfig,psfrag} 
   \usepackage{paralist,subfigure}
   \usepackage[dvipsnames]{xcolor}
   \usepackage{hyperref} 

    \usepackage[active]{srcltx} 
    \usepackage{color,soul} 
    \usepackage[latin1]{inputenc}
    \usepackage[OT1]{fontenc}
    \usepackage{authblk}

\newtheorem{theorem}{Theorem}[section]
\newtheorem{cor}[theorem]{Corollary}
\newtheorem{lemma}[theorem]{Lemma}
\newtheorem{prop}[theorem]{Proposition}

\theoremstyle{definition}
\newtheorem{definition}[theorem]{Definition}
\newtheorem{req}{Remark}

\newtheorem{question}{Question}

\def\N{\mathbb{N}}

\def\R{\mathbb{R}}

\let\O=\Omega
\let\e=\varepsilon

\let\t=\tilde
\let\ol=\overline
\let\ul=\underline
\let\.=\cdot
\let\0=\emptyset

\let\mc=\mathcal

\def\1{\mathbbm{1}}

\def\l{\lambda_1}

\newenvironment{formula}[1]{\begin{equation}\label{#1}}
                       {\end{equation}\noindent}

\def\Fi#1{\begin{formula}{#1}}
\def\Ff{\end{formula}\noindent}

\setstcolor{red}

\title{\bf Influence of a road on a population in an ecological niche facing climate change}

\author[1]{Henri {\sc Berestycki}}
\author[1,2]{Romain {\sc Ducasse}}
\author[1]{Luca {\sc Rossi}}
\affil[1]{Ecole des Hautes Etudes en Sciences Sociales,  
		Centre d'Analyse et Math\'ematiques Sociales, CNRS, 54 boulevard Raspail 
		75006 Paris, France}
\affil[2]{Institut Camille Jordan, Université Claude Bernard Lyon 1}			

\begin{document}

\date{}
\maketitle



\begin{abstract}

We introduce a model designed to account for the influence of a line with fast diffusion -- 
such as a road or another transport network -- on the dynamics
of a population in an ecological niche. This model consists of a system of coupled reaction-diffusion equations set on domains with different dimensions (line / plane).
We first show that, in a stationary climate, the presence of the line is always deleterious and can even lead the population to extinction. 
Next, we consider the case where the niche is subject to a displacement, 
representing the effect of a climate change or of 
seasonal variation of resources. 
We find that in such case the 
line with fast diffusion can help the population to persist. We also study several qualitative properties of this system. The analysis is based on a notion of \emph{generalized principal eigenvalue} developed and studied by the authors in \cite{BDR1}.
\end{abstract}

\section{Setting of the problem and main results}
\subsection{Introduction}

It has long been known that the spreading of invasive species can be enhanced by human transportations. 
This effect has become more pervasive because of the globalization of trade and transport. It has led to the introduction of some species very far from their originating habitat. This was the case for instance for the ``tiger mosquito", \emph{Aedes Albopictus}. Originating from south-east Asia, eggs of this mosquito were introduced in several places around the world, mostly via shipments of used tires, see~\cite{tires}. Though the tiger mosquito has rather low active dispersal capabilities, these long-distance jumps are not the only dispersion mechanisms involved in its spreading. Recently, there has been evidence of passive dispersal of adult tiger mosquitoes by cars, at much smaller scales, leading to the colonizations of new territories along road networks, see \cite{car, ML}.

Other kind of networks with fast transportation appear to help the dispersal of biological entities. For instance, rivers can accelerate the spreading of plant pathogens, cf. \cite{JB}. It has also been observed that populations of wolves in the Western Canadian Forest move and concentrate along seismic lines (paths traced in forests by oil companies for testing of oil reservoirs), cf. \cite{W1, W2}. In a different register, we mention that the road network is known to have a driving effect on the spreading of epidemics. The ``black death" plague, for instance, spread first along the silk road and then spread along the main commercial roads in Europe, cf. \cite{Si}.

 All these facts suggest that networks with fast diffusion (roads, rivers, seismic lines...) are important factors to take into account in the study of the spreading of species. A mathematical formulation of a model accounting for this phenomenon was introduced in \cite{BRR1} by the first and third author, in collaboration with J.-M. Roquejoffre. The width of the lines with fast diffusion being much smaller than the natural scale of the problem, the model introduced in \cite{BRR1} consists in a system of coupled reaction-diffusion equations set on domains of different dimensions, namely a line and the plane or half-plane. An important feature is that it is \emph{homogeneous}, in the sense that the environment does not change from one place to another.\\

This ``homogeneity" hypothesis does not hold in several situations. For instance, many observations suggest that the spreading of invasive species can happen only when the environment is ``favorable enough". Considering again the tiger mosquito, the climate is known to limit its range of expansion. In America, the tiger mosquito has reached its northernmost boundary in New Jersey, southern New York and Pennsylvania. It is believed that cold temperatures are responsible for stopping its northward progression. This means that the \emph{ecological niche} of the tiger mosquito is limited by the climate conditions. In this paper, we call an ecological niche a portion of the space where a population can reproduce, surrounded by an \emph{unfavorable} domain, lethal for the population. From a biological perspective, the niche can be characterized by a suitable temperature range, or by a localization of resources, for instance.

An important feature of an ecological niche is that it can move as time goes by. For instance, global warming raising the temperature to the north of the territory occupied by the tiger mosquito, leads to a displacement of its ecological niche. This should entail the further spreading of the mosquito into places that were unaccessible before, see \cite{rnhf}. The displacement of the ecological niche could also result from seasonal variation of resources. A.B. Potapov and M.A. Lewis \cite{PL}, and the first author of this paper together with O. Diekmann, P. A. Nagelkerke and C. J. Zegeling~\cite{BDNZ} have introduced a model designed to describe the evolution of a population facing a shifting climate. We review some of their results in the next section.
\\

In the present paper, we introduce and study a model of population dynamics which takes into account both these phenomena: it combines a line with fast diffusion and an ecological niche, possibly moving in time, as a consequence for instance of a climate change. Consistently with the existing literature on the topic, we will refer in the sequel to the line with fast diffusion as the ``road" and to the rest of the environment as ``the field". The two phenomena we consider are in some sense in competition: the road enhances the diffusion of the species, while the ecological niche confines its spreading. Two questions naturally arise.
\begin{question}\label{question 1}
Does the presence of the road help or, on the contrary, inhibit the persistence of the species living in an ecological niche?
\end{question}
\begin{question}\label{question 2}
What is the effect of a moving niche?
\end{question}
\subsection{The model}
The goal of this paper is to investigate these questions.
We consider a two dimensional model, where the road is the one-dimensional line $\R\times \{0\}$ and the field is the upper half-plane $\R\times \R^+$. Let us mention that we can consider as well a field given by the whole plane. This does not change the results presented here, as we explain in Section \ref{whole plane} below. However, the notations become somewhat cumbersome. We refer to \cite{BDR1}, where road-field systems on the whole plane are considered.

 As in \cite{BRR1}, we use two distinct functions to represent the densities of the population on the road and in the field respectively: $u(t,x)$ is the density on the road at time $t$ and point $(x,0)$, while $v(t,x,y)$ is the density of population in the field at time $t$ and point $(x,y) \in \R\times\R^+$.

In the field, we assume that the population is subject to diffusion, and also to reaction, accounting for reproduction and mortality. The presence of the ecological niche is reflected by an heterogeneous reaction term which is negative outside a bounded set. For part of our study, we also allow the niche to move with constant speed $c\in \R$.
On the road, the population is only subject to diffusion. The diffusions in the field and on the road are constant but \emph{a priori} different. Moreover, there are exchanges between the road and the field: the population can leave the road to go into the field and can enter the road from the field with some (a priori different) probabilities. 

Combining these definitions and effects, our system writes:
 \begin{equation}\label{frwcc}
\begin{cases}
\partial_{t}u-D\partial_{xx}u = \nu v\vert_{y=0}-\mu u , &\quad t >0,\ x\in \mathbb{R}, \\
\partial_{t}v-d\Delta v = f(x - ct,y,v), &\quad t >0 ,\ (x,y)\in \R\times\R^+, \\
- d\partial_{y}v\vert_{y=0}= \mu u-\nu v\vert_{y=0}, &\quad t >0 ,\ x\in \mathbb{R}.
\end{cases}
\end{equation}
%
The first equation accounts for the dynamic on the road, the second for the dynamic in the field, and the third for the exchanges between the field and the road. 
Note that the term $ \nu v\vert_{y=0}-\mu u$ represents the balance of the exchange between the road and the field (gained by the road and lost by the field). 
Unless otherwise specified, we consider classical solutions when dealing with the parabolic problem \eqref{frwcc}.
%
%
%
In the system \eqref{frwcc}, $D, d, \mu, \nu$ are strictly positive constants and $c$ is a real number. Without loss of generality, we consider only the case $c\geq0$, that is, the niche is moving to the right. The nonlinear
term $f$ depends on the variable $x-ct$. This implies that the spatial heterogeneities are shifted with speed $c$ in the direction of the road. We will first consider the case $c=0$ that is, when there is no shift.

Throughout the paper, besides some regularity hypotheses (see Section~\ref{Our model})
 we assume that $f(x,y,v)$ vanishes at $v=0$ (no reproduction occurs if there are no individuals) and that the environment has a maximal carrying capacity:
\begin{equation}\label{sat}
\exists S>0 \ \text{ such that } \ f(x,y,v) <0 \ \text{ for all } \ v \geq S,\ (x,y)\in \mathbb{R}\times(0,+\infty).
\end{equation}
We then assume that the \emph{per capita} net growth rate is a decreasing function of the size of the population, that is,
\begin{equation}\label{KPPCC}
v \mapsto \frac{f(x,y,v)}{v} \text{ is strictly decreasing for} \ v\geq0, \ (x,y) \in \mathbb{R}\times[0,+\infty). 
\end{equation}
In particular, $f$ satisfies the Fisher-KPP hypothesis: $f(x,y,v)\leq f_v(x,y,0)v$, for $v\geq0$.

Last, we assume that the ecological niche is~{\em bounded}\,:
\begin{equation}\label{BFZ}
\limsup_{\vert (x,y) \vert \to +\infty} f_{v}(x,y,0) <0.
\end{equation}
An example of a nonlinearity satisfying the above assumptions is $f(x,y,v) = \zeta(x,y)v-v^2$, with $\limsup_{\vert (x,y) \vert \to +\infty} \zeta(x,y) <0$.

To address Questions $1$ and $2$, we will compare the situation ``with the road" with the situation ``without the road". When there is no road, the individuals in the field who reach the boundary $\R\times\{0\}$ bounce back in the field instead of entering the road. In other terms, removing the road from system \eqref{frwcc} leads to the Neumann boundary problem
\begin{equation}\label{no road}
\left\{
\begin{array}{rll}
\partial_{t}v-d\Delta v &= f(x-ct,y,v), \quad &t>0,\ (x,y) \in   \mathbb{R}\times \mathbb{R}^{+}, \\
\partial_{y}v\vert_{y=0} &=0 , \quad&t>0,\ x \in  \mathbb{R}.
\end{array}
\right.
\end{equation}
By a simple reflection argument, it is readily seen that the dynamical properties of this system are the same as for the problem in the whole plane, as explained in Section~\ref{whole plane} below.

In the sequel, we call \eqref{frwcc} the system ``with the road", while \eqref{no road} is called the system ``without the road". Problem \eqref{no road} describes the evolution of a population subject to a climate change only. In the next section, we recall the basic facts on this model. 
%
%
%
%
%
%
%

Questions \ref{question 1} and \ref{question 2} translate in terms of the comparative dynamics of \eqref{frwcc} and~\eqref{no road}: we will see that, depending on the parameters, the solutions of these systems either asymptotically stabilize to a positive steady state, in which case the population persists, or vanish, meaning that the population goes extinct. Therefore, we will compare here the conditions under which one or the other of these scenarios occurs, for systems \eqref{frwcc} and \eqref{no road}.

\subsection{Related models and previous results}\label{previous models}

We present in this section some background about reaction-diffusion equations as well as the system from which \eqref{frwcc} is originated. These results will  be used in the~sequel.

Consider first the classical reaction-diffusion equation introduced by R. A. Fisher~\cite{F} and A. N. Kolmogorov, I. G. Petrovski and N. S. Piskunov \cite{KPP}:
\begin{equation}\label{eqKPP}
\partial_t v - d\Delta v = f(v), \quad  t>0, \ x\in \R^N,
\end{equation}
with $d>0$, $f(0)=f(1)=0$, $f(v)>0$ for $v \in (0,1)$ and $f(v)\leq f'(0)v$ for $v \in [0,1]$. The archetypical example is the logistic nonlinearity $f(v)=v(1-v)$. 
Under these conditions, we refer to \eqref{eqKPP} as the Fisher-KPP equation. It is shown in \cite{AW} that \emph{invasion} occurs for any nonnegative and not identically equal to zero initial datum. That is, any solution $v$ arising from such an initial datum converges to $1$ as $t$ goes to $+\infty$, locally uniformly in space. Moreover, if the initial datum has compact support, one can quantify this phenomenon by defining the \emph{speed of invasion} as a value $c_{KPP} > 0$ such that:
$$
\forall \, c > c_{KPP}, \quad \displaystyle{\sup_{\vert x \vert \geq ct}} v(t,x) \underset{t \to +\infty}{\longrightarrow} 0,
$$
and
$$
\forall \, c < c_{KPP}, \quad \displaystyle{\sup_{\vert x \vert \leq ct  }} \vert v(t,x) -1 \vert \  \underset{t \to +\infty}{\longrightarrow} 0.
$$
The speed of invasion can be explicitly computed in this case: $c_{KPP}=2\sqrt{d f^{\prime}(0)}$. 
\\
\\
Building on equation \eqref{eqKPP},  Potapov and Lewis \cite{PL} and H. Berestycki, O. Diekmann, C. J. Nagelkerke and P. A. Zegeling  \cite{BDNZ} proposed a model describing the effect of a climate change on a population in dimension $1$. 
H. Berestycki and L. Rossi in~\cite{BRforcedspeed} have further studied this model in higher dimensions and under more general hypotheses . It consists in the following reaction-diffusion equation
\begin{equation}\label{OriginalCC}
\partial_{t}v -d\Delta v = f(x-ct,y,v)  , \quad  t>0  ,\   (x,y)\in \R^2,
\end{equation}
with $f$ satisfying the same hypotheses \eqref{KPPCC} and \eqref{BFZ} extended to the whole plane. The favorable zone moves with constant speed $c$ in the $x$-direction. Let us mention that, if the nonlinearity $f$ is even with respect to the vertical variable, i.e., if $f(\cdot,y,\cdot)=f(\cdot,-y,\cdot)$ for every $y\in \R$, then equation \eqref{OriginalCC} is equivalent to the problem \eqref{no road} ``without the road", at least for solutions which are even in the variable $y$. It turns out that the results of \cite{BRforcedspeed} hold true for such problem, as we explain in details in Section \ref{whole plane} below.

 In the frame moving with the favorable zone, \eqref{OriginalCC} rewrites
%
%
\begin{equation}\label{OriginalCCmf}
\partial_{t}v -d\Delta v-c\partial_x v = f(x,y,v)  , \quad  t>0  ,\   (x,y) \in \R^2.
\end{equation}
The dependance of the nonlinear term in $t$ disappears and is replaced by a drift-term. From a modeling point of view, a drift term can also describe a stream or a wind, or any such transport. Intuitively, the faster the wind, the harder it would be for the population to stay in the favorable zone (that does not move in this frame). Hence, the faster the favorable zone moves, the harder it should be for the population to keep track with it. This intuition is made rigorous in \cite{PL, BDNZ, BRforcedspeed}, where the authors prove the following (for $x\in \R$, we write $[x]^+ := \max \{x ,0\}$).

\begin{prop}[{\cite{BRforcedspeed}}]\label{rappel CC}
There exists $c_{N} \geq 0$ such that
\begin{enumerate}[$(i)$]
\item If $0 \leq c < c_N$, there is a unique bounded positive stationary solution of \eqref{OriginalCCmf}, and any solution arising from a non-negative, not identically equal to zero, bounded initial datum converges to this stationary solution as $t$ goes to $+\infty$. 

\item If $c\geq c_N$, there is no bounded positive stationary solution of \eqref{OriginalCCmf} and any solution arising from a non-negative, not identically equal to zero
initial datum converges to zero uniformly as $t$ goes to $+\infty$. 

\end{enumerate}
\end{prop}
Our system \eqref{frwcc} is also inspired by the \emph{road-field model}, introduced by two of the authors with J.-M. Roquejoffre in \cite{BRR1}. They studied the influence of a line with fast diffusion on a population in an environment governed by a homogeneous Fisher-KPP equation. Their model reads
\begin{equation}\label{OriginalFR}
\begin{cases}
\partial_{t}u-D\partial_{xx}u = \nu v\vert_{y=0}-\mu u , &\quad t >0,\ x\in \mathbb{R}, \\
\partial_{t}v-d\Delta v = f(v), &\quad t >0 ,\ (x,y)\in \R\times\R^+, \\
- d\partial_{y}v\vert_{y=0}= \mu u-\nu v\vert_{y=0}, &\quad t >0 ,\ x\in \mathbb{R}.
\end{cases}
\end{equation}
%
The novelty in our system \eqref{frwcc} with respect to system \eqref{OriginalFR} is that we allow the nonlinearity to depend on space and time variables. The main result of \cite{BRR1} can be summarized as follows.
\begin{prop}[{\cite[Theorem 1.1]{BRR1}}]\label{rappel FR} 
Invasion occurs in the direction of the road for system \eqref{OriginalFR} with a speed $c_{H}$. That is, for any solution $(u,v)$ of \eqref{OriginalFR} arising from a compactly supported non-negative not identically equal to zero initial datum, there holds
$$
\forall h>0 , \ \forall  c<c_{H}, \quad
\sup_{\substack{ \vert x \vert \leq ct \\ \vert y \vert \leq h}} \vert v(t,x,y) -1\vert \underset{t\to +\infty}{\longrightarrow}  0 ,\quad \sup_{\substack{\vert x \vert \leq ct}}  \left\vert u(t,x) -\frac{\nu}{\mu} \right\vert \underset{t\to +\infty}{\longrightarrow}  0,
$$
 and
 $$
 \forall  c>c_{H}, \quad
\sup_{\substack{ \vert (x,y) \vert \geq ct }} v(t,x,y) \underset{t\to +\infty}{\longrightarrow}   0 ,\quad \sup_{\substack{\vert x \vert \geq ct}} u(t,x) \underset{t\to+ \infty}{\longrightarrow}  0.
$$
Moreover, $c_H \geq c_{KPP}$ and
$$
c_{H} > c_{KPP} \quad \text{ if and only if }\ D>2d.
$$
\end{prop}
Recall that $c_{KPP} = 2d\sqrt{f^{\prime}(0)}$ is the speed of invasion for \eqref{eqKPP}, that is, in the absence of a road. Hence, this result means that the speed of invasion in the direction of the road is enhanced, provided the diffusion on the road $D$ is large enough compared to the diffusion in the field $d$.

Several works have subsequently extended model \eqref{OriginalFR} in several ways. The article~\cite{BRR2} studies the influence of drift terms and mortality on the road. In a further paper~\cite{BRR3}, H. Berestycki, J.-M. Roquejoffre and L. Rossi \cite{BRR3} compute the spreading speed in all directions of the field. The paper \cite{GMZ} treats the case where the exchanges coefficients $\mu, \nu $ are not constant but periodic in $x$. The articles \cite{P1, P2} study non-local exchanges and \cite{Diet1} considers a combustion nonlinearity instead of the KPP one together with other aspects of the problem. The articles \cite{BCRR_sem,BCRR} study the effect of non-local diffusion. Different geometric situations are considered in \cite{RTV, Dcurved}. The first one treats the case where the field is a cylinder with its boundary playing the role of the road, and the second one studies the case where the road is curved.



%

\subsection{Main results}\label{Our model}

We assume in the whole paper that the nonlinearity $f$ is globally Lipschitz-continuous and that $v \mapsto f(x,y,v)$ is of class $C^1$ in a neighborhood of $0$, uniformly in $(x,y)$. The hypotheses~\eqref{sat},~\eqref{KPPCC} and~\eqref{BFZ} will also be understood to hold throughout the whole paper without further mention. For notational simplicity, we define
$$
m(x,y):= f_v(x,y,0).
$$
This is a bounded function on $\R\times \R^+$.

As in the case of the climate change model \eqref{OriginalCC}, it is natural to work in the frame moving along with the forced shift. There, the system ``with the road"~\eqref{frwcc}~rewrites
%
\begin{equation}\label{mf}
\left\{
\begin{array}{rll}
\partial_{t}u-D\partial_{xx}u - c\partial_{x}u&= \nu v\vert_{y=0}-\mu u, \quad &t>0,\ x \in  \mathbb{R}, \\
\partial_{t}v-d\Delta v -c\partial_{x}v &= f(x,y,v), \quad &t>0,\ (x,y) \in   \mathbb{R}\times \mathbb{R}^{+}, \\
-d\partial_{y}v\vert_{y=0} &= \mu u-\nu v\vert_{y=0}, \quad &t>0,\ x \in  \mathbb{R}.
\end{array}
\right.
\end{equation}
Likewise,  in the moving frame the system ``without the road" \eqref{no road} takes the form:

\begin{equation}\label{no road mf}
\left\{
\begin{array}{rll}
\partial_{t}v-d\Delta v -c\partial_{x}v &= f(x,y,v), \quad &t>0,\ (x,y) \in   \mathbb{R}\times \mathbb{R}^{+}, \\
-\partial_{y}v\vert_{y = 0} &= 0, \quad &t>0, \ x \in  \mathbb{R}.
\end{array}
\right.
\end{equation}
In this paper we investigate the long-time behavior of solutions of \eqref{mf} in comparison with \eqref{no road mf}. We will derive a dichotomy concerning two opposite scenarios: \emph{extinction} and \emph{persistence}.

\begin{definition}\label{def 1}

For the systems \eqref{mf} or \eqref{no road mf}, we say that

\begin{enumerate}[(i)]
\item \emph{extinction} occurs if every solution arising from a non-negative compactly supported initial datum converges uniformly to zero as $t$ goes to $+\infty$;

\item \emph{persistence} occurs if every solution arising from a non-negative not identically equal to zero compactly supported initial datum converges locally uniformly to a positive stationary solution as $t$ goes to $+\infty$.
\end{enumerate}
\end{definition}
We will show that the stationary solution, when it exists, takes the form of a \emph{traveling pulse}. In the original frame, it decays at infinity, due to the assumption~\eqref{BFZ}.

In the Fisher-KPP setting considered in this paper, it is natural to expect the phenomena of extinction and persistence to be characterized by the stability of the null state $(0,0)$, i.e., by the sign of the smallest eigenvalue $\lambda$ of the linearization of the system~\eqref{mf} around~$(u,v) = (0,0)$:

\begin{equation}\label{lsmf}
\left\{
\begin{array}{rll}
-D\partial_{xx}\phi - c\partial_{x}\phi& - [ \nu \psi\vert_{y=0}-\mu \phi] = \lambda \phi , \quad &x \in  \mathbb{R}, \\
-d\Delta \psi-c\partial_{x}\psi & - m(x,y)\psi = \lambda \psi, \quad &(x,y) \in   \mathbb{R}\times \mathbb{R}^{+}, \\
-d\partial_{y}\psi\vert_{y=0} &= \mu \phi-\nu \psi\vert_{y=0}, \quad &x \in  \mathbb{R}.
\end{array}
\right.
\end{equation}
The smallest eigenvalue of an operator, when associated with a positive eigenfunction, is called the \emph{principal eigenvalue}. When dealing with operators that satisfy some compactness and monotonicity properties, the existence of the principal eigenvalue can be deduced from the Krein-Rutman theorem, see \cite{KR}. However, the system~\eqref{lsmf} is set on an unbounded domain. Hence, the Krein-Rutman theorem does not directly apply. Therefore, we make use in this paper of a notion of \emph{generalized principal eigenvalue},
in the spirit of the one introduced by 
H.~Berestycki, L.~Nirenberg and S.~Varadhan~\cite{BNV} 
to deal with elliptic operators on non-smooth bounded domains under Dirichlet conditions.
The properties of this notion have been later extended by 
H.~Berestycki and L.~Rossi \cite{BR4} to unbounded domains. 
The authors of the present paper introduced a notion of generalized principal eigenvalue adapted for road-field systems in \cite{BDR1}, that we will use here. 

In the sequel, $\l \in \R$ denotes the generalized principal eigenvalue of \eqref{lsmf}. Its precise definition is given in \eqref{fgpe} below and we also recall in Section \ref{recall} its relevant properties. Our first result states that indeed the sign of $\l$ characterizes the long-time behavior of \eqref{mf}.
Namely, there is a dichotomy between persistence and extinction given in Definition~\ref{def 1}, that is completely determined by the sign of~$\lambda_{1}$.
\begin{theorem}\label{persistence}
Let $\lambda_{1}$ be the generalized principal eigenvalue of system \eqref{lsmf}. 

\begin{enumerate}[(i)]

\item If $\lambda_{1}<0$,  system \eqref{mf} admits a unique positive bounded stationary solution and there is persistence.

\item If $\lambda_{1}\geq0$,  system \eqref{mf} does not admit any positive stationary solution and there is extinction.

\end{enumerate}
\end{theorem}
A result analogous to Theorem \ref{persistence} holds true for the system without the road~\eqref{no road mf}, with $\l$ replaced by the corresponding generalized principal eigenvalue, see Proposition \ref{persistence no road} below. This is a consequence of the results of~\cite{BRforcedspeed}, owing to the equivalence of behaviors in \eqref{no road mf} and \eqref{OriginalCCmf}, that we explain in Section \ref{whole plane} below.

Questions \ref{question 1} and \ref{question 2} are then tantamount to understanding the relation between the generalized principal eigenvalues associated with models \eqref{mf} and \eqref{no road mf}, and to analyze their dependance with respect to the parameters.

To this end, it will sometimes be useful to quantify the ``size" of the favorable zone by considering terms $f$ given by
\begin{equation}\label{f l}
f^{L}(x,y,u) := \chi(\vert (x,y) \vert -L)u -u^{2}.
\end{equation}
Here, $L \in \R$ represents the scale of the favorable region and $\chi$ is a smooth, decreasing function satisfying
$$
\chi(r) \underset{r \to -\infty}{\longrightarrow} 1
\ \text{ and }\ 
\chi(r) \underset{r \to +\infty}{\longrightarrow} -1.
$$
The nonlinearity $f^{L}$ satisfies both the Fisher-KPP condition \eqref{KPPCC} and the ``bounded favorable zone" hypothesis \eqref{BFZ}. Consistently with our previous notations, we define
$$
m^L(x,y) := f_v^L(x,y,0) = \chi( \vert (x,y)\vert~-~L).
$$
In this case, the favorable zone is the ball of radius $L+\chi^{-1}(0)$ (intersected with the upper half-plane), which is empty for $L\leq -\chi^{-1}(0)$. The fact that the favorable zone is a half-ball does not play any role in the sequel, and one could envision more general conditions.

We will first consider the case where $c=0$, that is, when the niche is not moving (there is no climate change).
\begin{theorem}\label{th road lethal}
Assume that $c=0$.
\begin{enumerate}[(i)]

\item Whatever the values of the parameters~$D,\mu,\nu$ are, if extinction occurs for the system ``without the road" \eqref{no road mf}, then extinction also occurs for the system ``with the road" \eqref{mf}.

\item When $f=f^L$, there exist some values of the parameters $d, D,L, \mu, \nu$ for which persistence occurs for the system ``without the road" \eqref{no road mf} while extinction occurs for the system ``with the road" \eqref{mf}.

\end{enumerate}

\end{theorem}
Theorem \ref{th road lethal} answers Question \ref{question 1}. Indeed, statement $(i)$ means that the presence of the road can never entail the persistence of a population which would be doomed to extinction without the road. In other words, the road never improves the chances of survival of a population living in an ecological niche. Observe that this result was not obvious a priori: first, there is no death term on the road, so the road is not a lethal environment. Second, if  the favorable niche were made of, say, two connected components, one might have thought that a road ``connecting" them might have improved the chances of persistence. Statement $(i)$ shows that this intuition is not correct.

Statement $(ii)$ asserts that the road can actually make things worse: there are situations where the population would persist in an ecological niche, but the introduction of a road drives it to extinction. This is due to an effect of ``leakage''of the population due to the road.

In the context of this result, we can discuss the roles 
of the diffusion parameters $d$ and $D$, that represent the amplitudes of the random motion of 
individuals in the field and on the road.

\begin{theorem}\label{d infinity}
Consider the system ``with the road" \eqref{mf}, with $c=0$. Then, there exists $d^{\star}\geq0$ depending on $D,\mu, \nu$ such that persistence occurs if and only if $0 <d<d^{\star}$. In particular, extinction occurs for every $d>0$ when $d^\star=0$. 
\end{theorem}
This result is analogous to the one discussed for the model without road in the one dimensional case in \cite{PL, BDNZ}. The interpretation is that the larger $d$, the farther the population will scatter away from the favorable zone, with a negative effect for persistence. Observe that when $d^\star =0$, then persistence never occurs (the set $(0,d^{\star})$ is empty). This is the case if there is no favorable niche at all. However, $d^\star>0$ as soon as $f>0$ somewhere. It is natural to wonder if a result analogous to Theorem~\ref{d infinity} holds true when, instead of the diffusion $d$ in the field, one varies the diffusion on the road $D$, keeping $d>0$ fixed. We show that this is not the case because there are situations where persistence occurs for all values of $D>0$, see Proposition \ref{D no influence} below.

Next, we turn to Question \ref{question 2}, that is, we consider the case $c > 0$ corresponding to a moving  niche caused e.g. by a climate change. We start with analyzing the influence of~$c$ on the 
survival of the species for the system ``with the road"~\eqref{mf}.
Owing to Theorem~\ref{persistence}, this amounts to studying the generalized principal eigenvalue $\lambda_1$ as a function of~$c$.
\begin{theorem}\label{critical}
There exist $0\leq c_{\star} \leq c^{\star}\leq2\sqrt{\max\{d,D\}  [\sup m]^+}$, 
such that the following holds for the system \eqref{mf}:
\begin{enumerate}[(i)]

\item Persistence occurs if $0 \leq c < c_{\star}$.

\item Extinction occurs if $ c \geq c^{\star}$.

\end{enumerate}
Moreover, if persistence occurs for $c=0$ then $c_{\star}>0$.
\end{theorem}

The quantities $c_{\star}$ and $c^{\star}$ are called the
\emph{lower} and \emph{upper critical speeds} respectively for~\eqref{mf}.
Theorem \ref{critical} has a natural interpretation: on the one hand, if $c$  is large, the population cannot keep pace with the moving favorable zone, and extinction occurs. On the other hand, if persistence occurs in the absence of climate change, it will also be the case with a climate change with moderate speed.

We do not know if the lower and upper critical speeds actually always coincide, that is, if $c_{\star} = c^{\star}$. We prove that this is the case when $d=D$, but we leave the general case as an open question.

Finally, we investigate the consequences of the presence of a road for a population facing a climate change. To this end, we focus on the case where $f=f^L$, given by~\eqref{f l}. Observe that formally, as $L$ goes to $+\infty$, \eqref{mf} reduces to the system~\eqref{OriginalFR} considered in \cite{BRR1}, in the same moving frame. This suggests that the critical speeds for \eqref{mf} should converge to the spreading speed $c_H$ of Proposition \ref{rappel FR}.
The next result makes this intuition rigorous.

%

\begin{theorem}\label{effect road}
Assume that $f = f^{L}$ in \eqref{mf}. Then, the lower and upper critical speed $c_\star, c^\star$ satisfy:
$$
c_{\star}, c^{\star} \nearrow c_H \quad \text{ as }\ L \nearrow +\infty,
$$
where $c_H$ is given by Proposition \ref{rappel FR}.
\end{theorem}

The above theorem has the following important consequence.

\begin{cor}\label{cor effect road}

Assume that $D>2d$. There are $L>0$ and $0<c_1<c_2$ such that
\begin{enumerate}[(i)]

\item If $c\in [0, c_1)$, persistence occurs for the model ``with the road" \eqref{mf} as well as for the model ``without the road" \eqref{no road mf}.

\item If $c\in [c_1,c_2)$, persistence occurs for the model ``with the road" \eqref{mf} whereas extinction occurs for the model ``without the road" \eqref{no road mf}.

\end{enumerate}

\end{cor}

This result answers Question $2$. Indeed, it means that, in some cases, the road can help the population to survive faster climate change than it would if there were no road. The threshold $D>2d$ in the theorem is the same threshold derived in \cite{BRR1} for the road to induce an enhancement of the asymptotic speed of spreading.
\\

The paper is organized as follows. In Section \ref{recall}, we recall some results from~\cite{BDR1}, concerning the generalized principal eigenvalue for system \eqref{lsmf}. We explain in Section \ref{whole plane} why the systems on the half plane are equivalent to the systems on the whole plane. We prove Theorem~\ref{persistence} in Section \ref{proof persistence}.
In Section \ref{section niche}, we study the effect of a road on an ecological niche, i.e., we consider~\eqref{mf} with $c=0$. We prove Theorems~\ref{th road lethal} and  \ref{d infinity} in Sections \ref{road lethal} and \ref{diffusions} respectively.
Section \ref{section road} deals with the effect of a road on a population facing climate change, i.e., system \eqref{mf} with $c >0$. We prove Theorem~\ref{critical} in Section \ref{section c} and Theorem~\ref{effect road} in Section \ref{section road helps}. Section \ref{unbounded niche} contains a brief discussion of the case where the favorable niche is not bounded, that is, when \eqref{BFZ} does not hold. This situation is still by and large open.

\section{The generalized principal eigenvalue and the long-time behavior} \label{recall}

\subsection{Definition and properties of the generalized principal eigenvalue}\label{sec rapp vp}

In this section, we recall some technical results from \cite{BDR1} concerning $\l$, the generalized principal eigenvalue for system \eqref{lsmf}. To simplify notations, we define the following linear operators:
\begin{equation}\label{notation}
\left\{
\begin{array}{lllrl}
\mc L_{1}(\phi,\psi) &:= D\partial_{xx}\phi+c\partial_{x}\phi +\nu \psi\vert_{y=0}-\mu \phi, \\
\mc L_{2}(\psi) &:= d\Delta \psi+c\partial_{x}\psi+m(x,y)\psi , \\
\mc B(\phi,\psi) &:= d\partial_{y}\psi\vert_{y=0}+\mu \phi-\nu \psi\vert_{y=0} .
\end{array}
\right.
\end{equation}
These operators are understood to act on functions $(\phi,\psi) \in W^{2,p}_{loc}(\R)\times W^{2,p}_{loc}(\R \times [ 0, +\infty)) $. 
We restrict to $p>2$, in order to have the imbedding in $C^{1}_{loc}(\R)\times C^{1}_{loc}(\R \times [ 0, +\infty))$.

The generalized principal eigenvalue of \eqref{lsmf} is defined by
\begin{equation}\label{fgpe}
\begin{array}{rl}
\lambda_{1}:=\sup\Big\{ \lambda \in \mathbb{R} \, : \, \exists (\phi,\psi) \geq 0, &(\phi,\psi)\not\equiv (0,0) \, \text{such that} \,
\mc L_{1}(\phi,\psi)+\lambda \phi \leq 0 \ \text{ on }\, \R ,\\  &\mc L_{2}(\psi)+\lambda \psi  \leq 0  \ \text{ on }\ \R\times\R^+,\  \mc B(\phi,\psi) \leq 0\, \text{ on }\, \R \Big\}.
\end{array}
\end{equation}
Above and in the sequel, unless otherwise stated, the differential equalities and inequalities are understood to hold almost everywhere.
%
%
%

Owing to Theorem \ref{persistence}, proved below in this section, the sign of $\l$ completely characterizes the long-time behavior for system \eqref{mf}. Therefore, to answer Questions \ref{question 1} and \ref{question 2}, and the other questions addressed in this paper, we will study the dependence of $\l$ with respect to the coefficients $d,D,c$ as well as with respect to the parameter $L$ in \eqref{f l}. The formula \eqref{fgpe} is not always easy to handle, but there are two other characterizations of $\l$ which turn out to be handy. First, $\l$ is the limit of principal eigenvalues of the same problem restricted to bounded domains that converge to the half-plane. More precisely, calling $B_R$ the (open) ball of radius $R$ and of center $(0,0)$ in $\R^2$, we consider the increasing sequences of (non-smooth) domains $(\O_R)_{R>0}$ and $(I_R)_{R>0}$ given by
$$
\O_R := B_R \cap (\R \times \R^+) \quad \text{and} \quad I_R = (-R,R).
$$
%
%
%
We introduce the following eigenproblem: 
\begin{equation}\label{eigborn}
\left\{
\begin{array}{rll}
-\mc L_{1}(\phi,\psi) &= \lambda \phi \  &\text{ in } I_R, \\
-\mc L_{2}(\psi) &= \lambda \psi \ &\text{ in }\  \Omega_{R}, \\
\mc B(\phi,\psi) &= 0 \ &\text{ in }  \ I_R, \\
\psi &= 0 \ &\text{on} \   (\partial \Omega_{R})\backslash (I_R\times\{0\}), \\
\phi(-R)=\phi(R)&=0 .
\end{array}
\right.
\end{equation}
%
Here, the unknowns are $\lambda \in \R$, $\phi \in W^{2,p}(I_R)$ and $\psi \in W^{2,p}(\O_R)$. The existence of a principal eigenvalue and its connection with the generalized principal eigenvalue are given by the next result.

\begin{prop}[{\cite[Theorem 2.2]{BDR1}}]\label{limit}
For $R>0$, there is a unique $\lambda_R \in \R$ and a unique (up to multiplication by a positive scalar) positive pair $(\phi_R,\psi_R) \in W^{2,p}(I_R)\times W^{2,p}(\O_R)$ that satisfy \eqref{eigborn}. 

Moreover, the following limit holds true
$$
\lambda_{1}^{R} \underset{R \to +\infty}{\searrow} \lambda_{1}.
$$

Finally, there is a generalized principal eigenfunction associated with $\l$, that is, a pair $(\phi,\psi)~\in~W^{2,p}_{loc}(\R)\times W^{2,p}_{loc}(\R \times [ 0, +\infty))$, $(\phi,\psi) \geq 0$, $(\phi,\psi)\not\equiv (0,0)$ satisfying $\mc L_1(\phi,\psi) = \l \phi $, $\mc L_2(\psi) = \l \psi$ and $\mc B(\phi,\psi) =0$.
\end{prop}
We refer the reader to \cite{BDR1} for the details. The real number $\l^R$ and the pair $(\phi_R,\psi_R)$ are called respectively the principal eigenvalue and eigenfunction of \eqref{eigborn}.

The next characterization of $\l$ is obtained in the case when $c=0$. It is 
in the spirit of the classical Rayleigh-Ritz formula. We introduce the following Sobolev space:
$$
\t H^1_{0}(\O_R) := \{ u \in H^1(\O_R) \ : \ u = 0 \ \text{ on } (\partial B_R)\cap (\R\times \R^+) \ \text{in the sense of the trace}  \}.
$$

\begin{prop}[{\cite[Proposition 4.5]{BDR1}}]\label{variational}
Assume that $c=0$. The principal eigenvalue $\lambda_{1}^{R}$ of \eqref{eigborn} satisfies
\begin{equation}\label{eqvariational}
\lambda_{1}^{R} = \inf_{\substack{(\phi,\psi) \in \mathcal{H}_{R} \\ (\phi,\psi) \not\equiv (0,0)}} \frac{  \mu\int_{I_R}D \vert \phi^{\prime} \vert^{2}  + \nu\int_{\Omega_{R}} \left( d\vert \nabla \psi \vert^{2} - m\psi^{2} \right) + \int_{I_R}(\mu \phi- \nu \psi\vert_{y=0})^{2}  }{\mu\int_{I_R}\phi^{2} + \nu\int_{\Omega_{R}}\psi^{2} },
\end{equation}
where we recall that $m=f_v(\cdot,\cdot,0)$, and
$$
\mathcal{H}_{R} :=   H^{1}_{0}(I_R) \times \t H^1_{0}(\O_R)
$$

\end{prop}

%
%
%
%
%
%
%
%
%
Let us also recall the following result concerning the continuity and monotonicity of~$\l$. We use the notation $\l(c,L,d,D)$ to indicate the generalized principal eigenvalue of \eqref{lsmf}, with coefficients $c,d,D$ and with nonlinearity $f^L$ given by \eqref{f l}. Then, we treat $\l$ as a function from $(\R)^2\times(\R^+)^2$ to $\R$. Analogous notations will be used several times in the sequel.

\begin{prop}[{\cite[Propositions 2.4 and 2.5]{BDR1}}]\label{prop function}
Let $\lambda_{1}(c,L,d,D)$ be the generalized principal eigenvalue of system~\eqref{lsmf} with nonlinearity $f^L$ defined by \eqref{f l}. Then,

\begin{itemize}

\item $\l(c,L,d,D)$ is a locally Lipschitz-continuous function on $(\R)^2\times(\R^+)^2$.

\item If $c=0$, then $\l(c,L,d,D)$ is non-increasing with respect to $L$ and non-decreasing with respect to $d$ and $D$.

\item If $c=0$ and $\l(c,L,d,D)\leq 0$, then $\l(c,L,d,D)$ is strictly decreasing with respect to $L$ and strictly increasing with respect to $d$ and $D$.

\end{itemize}

\end{prop}

%

Next, we consider the generalized principal eigenvalue for the model ``without the road"  \eqref{no road mf}, that we require to answer Questions \ref{question 1} and~\ref{question 2}. The linearization around $v =0$ of the stationary system associated with \eqref{no road mf} reads
\begin{equation}\label{lsmfn}
\left\{
\begin{array}{rll}
-\mc L_2(\psi) &=0, \quad &(x,y) \in   \mathbb{R}\times \mathbb{R}^{+}, \\
-\partial_{y}\psi\vert_{y=0} &= 0, \quad &x \in  \mathbb{R},
\end{array}
\right.
\end{equation}
where $\mc L_2$ is defined in \eqref{notation}.
%
%
%
The generalized principal eigenvalue of \eqref{lsmfn} is given~by

\begin{equation}\label{fgpe no road}
\begin{array}{rl}
\lambda_{N}:=\sup\Big\{ \lambda \in \mathbb{R} \ : \ \exists  \psi \geq 0,\ \psi\not\equiv 0  
  &\text{ such that } \\
 (\mc L_{2}(\psi)+\lambda \psi ) &\leq 0 \ \text{ on }\ \R\times \R^+ , \   \partial_{y}\psi\vert_{y=0} \leq 0 \ \text{ on }\ \R \Big\}.
\end{array}
\end{equation}
The subscript $N$ refers to the Neumann boundary condition. Again, the test functions $\psi$ in \eqref{fgpe no road} are assumed to be in $W^{2,p}_{loc}(\R \times [0,+\infty))$. We also consider the principal eigenvalue on the truncated domains $\Omega_{R}$, which is the unique quantity $\lambda_N^R$ such that the problem

\begin{equation}\label{eigborn no road}
\left\{
\begin{array}{rllrl}
-\mc L_{2}(\psi) &= \lambda^{R}_{N}\psi, \quad  &(x,y) \in \Omega_{R}, \\
-\partial_{y}\psi\vert_{y=0} &= 0 , \quad &x \in I_R, \\
\psi(x,y) &= 0 , \quad  &(x,y) \in (\partial \Omega_{R}) \backslash {(I_R\times \{0\})},
\end{array}
\right.
\end{equation}
admits a positive solution $\psi\in W^{2,p}(\Omega_{R})$. The results concerning $\l$ hold true for $\lambda_N$. We gather them in the following proposition.
\begin{prop}\label{prop lambda n}
Let $\lambda_{N}$ be the generalized principal eigenvalue of the model ``without the road" \eqref{lsmfn}, and let $\lambda_{N}^R$ be the principal eigenvalue of \eqref{eigborn no road}. Then
\begin{itemize}
\item $\lambda_N$ is the decreasing limit of $\lambda_{N}^R$, i.e.,
\begin{equation}\label{cv eig no road}
\lambda_{N}^{R}\underset{R \to +\infty}{\searrow} \lambda_{N}.
\end{equation}
\item If $c=0$, then
\begin{equation}\label{eqvariational no road}
\lambda_{N}^{R} = \inf_{\psi \in \t H^1_{0}(\O_R)} \frac{  \int_{\Omega_{R}} \left( d\vert \nabla \psi \vert^{2} - m\psi^{2} \right)  }{\int_{\Omega_{R}}\psi^{2} }.
\end{equation}
\item If the nonlinearity in \eqref{lsmfn} is given by $f^L$, defined in \eqref{f l}, then $L\mapsto \lambda_{N}(L)$ is a continuous and non-increasing function.

\end{itemize}
\end{prop}
The two first points are readily derived from \cite{BRforcedspeed}: indeed, $\lambda_N$ coincides with the generalized principal eigenvalue of the problem in the whole space \eqref{OriginalCCmf} with~$f$ extended by symmetry, as explained in the next Section \ref{whole plane}. The third point concerning the monotonicity and the continuity comes from \cite{BR4}.

\subsection{The case of the whole plane}\label{whole plane}

Systems \eqref{mf} and \eqref{no road mf} are set on half-planes. Let us explain here how these models are actually equivalent to the same systems set on the whole plane under a symmetry hypothesis on the nonlinearity. When writing a road-field system where the road is not the boundary of an half-plane but a line in the middle of a plane, one needs to consider $3$ equations: one equation for each portion separated by the road and an equation on the road, completed with two exchanges conditions between the road and each side of the field. We assume that the exchanges are the same between the road and the two sides of the field. Moreover, we assume that the environmental conditions are symmetric with respect to the road, that is, the nonlinearity $f$ on the field is even with respect to the $y$ variable, i.e, $f(x,y,v) = f(x,-y,v)$ for every $(x,y) \in \R^2$ and $v \geq 0$. The system then writes (in the moving frame that follows the climate change):
 \begin{equation}\label{two fields}
\left\{
\begin{array}{rll}
\partial_{t}u-D\partial_{xx}u -c\partial_x u &= \nu ( v\vert_{y=0^{+}} +v\vert_{y=0^{-}})-\mu u , &\quad t >0,\ x\in \mathbb{R}, \\
\partial_{t}v-d\Delta v - c\partial_x v &= {f}(x,y,v), &\quad t >0 ,\ (x,y)\!\in\!\mathbb{R}
\!\times\!\mathbb{R}^{+} , \\
\mp d\partial_{y}v\vert_{y=0^{\pm}} &= \frac{\mu}{2} u-\nu v\vert_{y=0^{\pm}}, &\quad t >0 ,\ x\in \mathbb{R}.
\end{array}
\right.
\end{equation}
We point out that the set in the second equation has two connected components~and thus it can be treated as two distinct equations. The last line in \eqref{two fields} are also two equations with the proportion $\mu$ of $u$ leaving the road evenly split among the~two~sides.

Under these hypotheses of symmetry, the dynamical properties of the system~\eqref{two fields} are the same as those of the system on the half-plane \eqref{mf}. This is clear if one restricts to a symmetric initial datum $(u_0,v_0)$, i.e., such that $v_0(x,y)=v_0(x,-y)$ for every $(x,y)\in \R^2$. Indeed, the corresponding solution $(u,v)$ of \eqref{two fields} also satisfies $v(t,x,y) = v(t,x,-y)$ for every $t>0$, $(x,y) \in \R^2$, hence
$$
\left\{
\begin{array}{rll}
\partial_{t}u-D\partial_{xx}u -c\partial_x u &= 2\nu v\vert_{y=0^+} -\mu u , &\quad t >0,\ x\in \mathbb{R}, \\
\partial_{t}v-d\Delta v - c\partial_x v &= f(x,y,v), &\quad t >0 ,\ (x,y)\in \R\times \R^+, \\
- d\partial_{y}v\vert_{y=0^{\pm}} &= \frac{\mu}{2} u-\nu v\vert_{y=0^+}, &\quad t >0 ,\ x\in \mathbb{R}.
\end{array}
\right.
$$
It follows that $(\t u, v) := (\frac{1}{2}u,v)$ is a solution of the system with the road in the half-plane \eqref{mf}. For non-symmetric solutions of  \eqref{two fields}, the long-time behavior also turns out to be governed by \eqref{mf}. Indeed, any solution of \eqref{two fields} arising from a non zero initial datum is strictly positive at time $t=1$, and can then be nested between two symmetric solutions, both converging to the same stationary solution.

Let us mention that, in the paper \cite{BDR1}, where we define and study the notion of generalized principal eigenvalues for road-field systems, we also consider the case of non-symmetric fields.

By the same arguments as above, the problem without the road in the half-plane~\eqref{no road mf} 
is also seen to share the same dynamical properties as the equation in the whole plane~\eqref{OriginalCCmf}. 
Actually, a stronger statement holds true concerning the linearized stationary equations
\begin{equation}\label{stat plan}
-\mc L_2(\psi) = 0 , \quad (x,y) \in \R^2,
\end{equation}
without any specific assumptions on $m$ (besides regularity).
\begin{lemma}\label{lemma lambda N}
Assume that the nonlinearity $f(x,y,v)$ in \eqref{OriginalCCmf} is even with respect to the variable $y$.
Then extinction (resp.~persistence) occurs for \eqref{OriginalCCmf} if and only if it occurs for \eqref{no road mf}.

Moreover, \eqref{stat plan} admits a positive supersolution (resp.~subsolution) if and only if~\eqref{lsmfn} does. 
\end{lemma}
\begin{proof}
We have explained before that the long-time behavior for~\eqref{OriginalCCmf}
can be reduced to the one for~\eqref{no road mf}, that is, the first statement of the lemma
holds.

For the second statement, consider a positive supersolution $\omega$  of \eqref{stat plan}. 
We have~$\mc L_2(\omega)  \leq 0$ on $\R^2$, and this inequality also holds true for $\t \omega(x,y) := \omega(x,-y)$. Hence, the function $\psi := \omega + \t \omega$ satisfies  $\mc L_{2}(\psi) \leq 0$ on $\R\times \R^+$ and $\partial_{y}\psi\vert_{y=0} = 0$, i.e., $\psi$ is a positive supersolution for \eqref{lsmfn}.

Take now a positive supersolution $\psi$ of \eqref{lsmfn}, that is, 
$\mc L_2(\psi) \leq 0$ on $\R\times \R^+$ and $\partial_y \psi\vert_{y=0}\leq 0$
on $\R$. One would like to use the function $\psi(x,\vert y \vert)$ as a supersolution for~\eqref{stat plan}, however this function is not in $W^{2,p}(\R^2)$. To overcome this difficulty, define $\t \psi(t,x,y)$ to be the solution of the parabolic problem
\begin{equation*}
\partial_t \t \psi = \mc L_2(\t\psi),\quad  t>0,\ (x,y)\in \R^2,
\end{equation*}
with initial datum $\psi(x,\vert y\vert)$. The function $\t \psi(1,x,y)$ is positive and is in $W^{2,p}(\R^2)$. Moreover the parabolic comparison principle yields $\partial_t \t \psi \leq 0$, hence the function $\omega(x,y) := \t \psi(1,x,y)$ satisfies $\mc L_2(\omega) \leq 0$ on $\R^2$, i.e., it is a positive supersolution to \eqref{OriginalCCmf}.
\end{proof}
The second part of Lemma \ref{lemma lambda N} applied to $\mc L_2 + \lambda$ implies that the operator $\mc L_2$ set on $\R\times\R+$ with Neumann boundary conditions and the operator $\mc L_2$ set on $\R^2$ share the same generalized principal eigenvalue, that is:
$$
\lambda_{N}  = \sup\Big\{ \lambda \in \mathbb{R} \ : \ \exists  \psi \geq 0,\ \psi\not\equiv 0  
 \text{ such that } \ 
 (\mc L_{2}(\psi)+\lambda \psi ) \leq 0 \ \text{ on }\ \R^2 \Big\}.
$$

\subsection{The long-time behavior for the system with the road}\label{proof persistence}

This section is dedicated to proving Theorem \ref{persistence}. We first derive in Section \ref{SectionLiouville} a Liouville-type result, namely we show that there is at most one non-negative, not identically equal to zero, bounded stationary solution of the semilinear system \eqref{mf}. This will be used in Section~\ref{SectionExistence} to characterize the asymptotic behavior of solutions of the evolution problem \eqref{mf} in terms of the generalized principal eigenvalue $\l$.

	For future convenience, let us state  the parabolic strong comparison principle 
	 for the road-field 
	system~\eqref{mf}. This is derived in \cite[Proposition 3.2]{BRR1} with~$f$ independent
	of $x,y$, but the proof does not change if one adds this dependence.
	
	We say that a pair $(u,v)$ is a supersolution (resp. subsolution) of \eqref{mf} if it solves \eqref{mf} with all the signs $=$ replaced by $\geq$ (resp. $\leq$).
	\begin{prop}\label{pro:cp}
		Let $(u_1,v_1)$ and $(u_2,v_2)$ be respectively a bounded sub and supersolution  
		of~\eqref{mf} such that  $(u_1,v_1)\leq(u_2,v_2)$ at time $t=0$.
		Then $(u_1,v_1)\leq(u_2,v_2)$ for all $t>0$, and the inequality is strict unless they 
		coincide until some $t>0$.	
	\end{prop}	
	\begin{req}\label{req:cp}
		The previous comparison principle applies in particular to stationary 
		sub and supersolutions, providing the strong comparison principle for the elliptic system associated to \eqref{mf}. Namely, if a stationary subsolution touches
		from below a stationary supersolution then they must coincide everywhere.
	\end{req}



\subsubsection{A Liouville-type result}\label{SectionLiouville}

We derive here the uniqueness of stationary solutions for \eqref{mf}.

\begin{prop}\label{Liouville}
There is at most one non-null bounded positive stationary  solution of \eqref{mf}.
\end{prop}
Before turning to the proof of Proposition \ref{Liouville}, we state a technical lemma.

\begin{lemma}\label{dec}
Let $(u,v)$ be a solution of the evolution problem~\eqref{mf} arising from a bounded non-negative initial datum. Then
$$
\sup_{\substack{\vert x \vert \geq R\\ t \geq R} }u(t,x) \underset{R \to +\infty}{\longrightarrow}0 \ \text{ and } \  \sup_{\substack{\vert (x,y) \vert \geq R\\ t \geq R} }v(t,x,y)\underset{R \to +\infty}{\longrightarrow}0.
$$
\end{lemma}


\begin{proof}
We first show that the conclusion of the lemma holds for the component $v$. 
Assume by contradiction that there are $\e>0$ and two diverging sequences, $(t_n)_{n \in \N}$ 
in $(0,+\infty)$ and $( (x_n , y_n))_{\in \N}$ in $\R\times (0,+\infty)$, such that
$$
\liminf_{n \to +\infty}v(t_n , x_n, y_n) \geq \e.
$$
The idea is to consider the equations satisfied by some translations of $u,v$. We divide the discussion into two different cases.

%
%

\medskip
 \emph{First case: $( y_{n})_{n \in \N}$ is unbounded.}\\
Up to extraction of a subsequence, we assume that $y_{n}$ goes to $+\infty$ as $n$ goes to $+\infty$. 
We define the translated functions
$$
u_n  := u(\cdot +t_n , \cdot +x_n)\  \text{ and } \ v_{n}  := v(\cdot+t_n , \cdot +x_n , \cdot +y_n).
$$
Because $((x_n , y_n))_{n\in \mathbb{N}}$ diverges, the zone where $f$ is positive disappears in the limit. More precisely, \eqref{BFZ} and \eqref{KPPCC} yield that there is $K>0$ such that
\begin{equation}\label{f limit}
\limsup_{n \to +\infty} f(x+x_{n},y+y_{n},z)\leq -K z \quad 
\text{for}\  (x,y)\in\R^2,\ z \geq 0.
\end{equation}
%
%
The parabolic estimates (see, for instance, \cite[Theorems 5.2 and 5.3]{La}) and \eqref{f limit} yield that, up to another extraction, $v_{n}$ converges locally uniformly to some $v_{\infty}$, entire (i.e., defined for all $t\in \R$) subsolution of
\begin{equation}\label{eq subsol}
\partial_{t}v -d\Delta v-c\partial_{x}v +Kv = 0, \quad t \in\R , \ (x,y) \in \mathbb{R}^{2}.
\end{equation}
Moreover,
$$
v_{\infty}(0,0,0) \geq \e.
$$
Observe that, for any $A>0$, the space-independent function $w(t) := Ae^{-K t}$ is a solution of \eqref{eq subsol}. Because $v_{\infty}$ is bounded, we can choose $A$ large enough so that, for every $\tau\geq0$,
$$
v_{\infty}( -\tau, \cdot, \cdot)\leq w(0) . 
$$
The parabolic comparison principle for \eqref{eq subsol} yields
$$
 v_{\infty}(t - \tau ,\cdot , \cdot) \leq  Ae^{-K t} \quad \text{ for } \ t \geq 0, \ \tau \in \R.
$$
Choosing $t=\tau$, we get
$$
\e \leq v_{\infty}(0,0,0) \leq A e^{-K \tau}.
$$
Taking the limit $\tau \to +\infty$ yields a contradiction.

\medskip
\emph{Second case: $( y_{n} )_{n\in \N}$ is bounded.} \\
Up to a subsequence, we assume that $y_n$ converges to some $y_{\infty}\geq 0$
as $n$ goes to $+\infty$.  
We now define the translated functions with respect to $(t_n)_{n\in\N}$, $(x_n)_{n\in\N}$ only:
$$
u_n  := u(\cdot +t_n , \cdot +x_n)\  \text{ and } \ v_{n}  := v(\cdot+t_n , \cdot +x_n , \cdot).
$$
Arguing as in the first case, we find that $((u_{n},v_{n}))_{n \in \N}$ converges locally uniformly (up to a subsequence) as $n$ goes to $+\infty$ to $(u_{\infty},v_{\infty})$, entire subsolution of
\begin{equation}\label{eq 2}
\left\{
\begin{array}{rll}
\partial_{t}u -D\partial_{xx}u - c\partial_{x}u&= \nu v\vert_{y=0} - \mu u, \quad &t \in \R, \ x \in  \mathbb{R}, \\
\partial_{t}v-d\Delta v -c\partial_{x}v+M v&= 0, \quad &t \in \R, \ (x,y) \in   \mathbb{R}\times \R^+, \\
-d\partial_{y}v\vert_{y=0} &= \mu u-\nu v\vert_{y=0}, \quad &t \in \R, \ x \in  \mathbb{R}.
\end{array}
\right.
\end{equation}
Moreover, we have
$$
v_{\infty}(0,0,y_\infty)\geq \e.
$$
Unlike the previous case, one does not have suitable space-independent solutions of \eqref{eq 2}. Instead, we look for a supersolution of the form
$$
A(e^{-\eta t}, \gamma e^{-\eta t}(e^{-\beta y} +1)).
$$
An easy computation shows that for this to be a supersolution of \eqref{eq 2} for any $A>0$,
it is sufficient to choose the parameters $\eta, \gamma ,\beta>0$ so that
\begin{equation*}
\left\{
\begin{array}{rlr}
-\eta&\geq 2\nu \gamma  - \mu, \\
-\eta(e^{-\beta y}+1)-d\beta^{2}e^{-\beta y}+M(e^{-\beta y}+1) & \geq 0,\quad &\text{for all }\ y\geq 0,  \\
d\beta \gamma &\geq \mu -2\nu \gamma.
\end{array}
\right.
\end{equation*}
For $\eta<\mu$, we take $\gamma = \frac{\mu - \eta}{2\nu}>0, \ \beta = \frac{\eta}{d \gamma }$, so that the first and third inequalities are automatically satisfied. For the second inequality to hold true for $y\geq 0$, it is sufficient to have $-\eta-d\beta^2+M\geq0$. We can take $\eta$ 
sufficiently small so that the latter is fulfilled.

We now choose $A$ large enough so that $A\geq\sup u_{\infty}$,  $\gamma A\geq\sup v_{\infty}$.
A contradiction is reached by arguing as in the first case, with the difference that now we need to use the parabolic comparison principle for the full road-field system, Proposition \ref{pro:cp}.

We have shown that
\Fi{vto0}
\sup_{\substack{\vert (x,y) \vert \geq R\\ t \geq R} }v(t,x,y)\underset{R \to +\infty}{\longrightarrow}0.
\Ff
Let us now derive the result for $u$.
Assume by contradiction that there are $\e>0$ and two diverging sequences $(t_n)_{n\in \N}$ and $(x_n)_{n\in \N}$
such that
$$
\liminf_{n\to+\infty}u(t_n,x_n)>\e.
$$
Then, because of~\eqref{vto0}, for $n$ large enough
the third equation in \eqref{mf} gives us~that
$$
\partial_{y}v(t_n,x_n,0) \leq -\frac{\mu}{2d}\e.
$$
The parabolic estimates then provides a constant $C>0$ such that for, say, 
$y\in (0,1)$, there holds that
$$
v(t_n,x_n,y)\leq v_n(t_n,x_n,0) -\frac{\mu}{2d}\e y + Cy^2.
$$
From this, taking  $y>0$ small enough and then $n$ large enough, and using again~\eqref{vto0},
we deduce that $v(t_n,x_n,y)<0$, which is impossible, hence the contradiction.
\end{proof}
We now turn to the proof of Proposition \ref{Liouville}.

\begin{proof}[Proof of Proposition \ref{Liouville}] 
Let $(u, v)$ and $(\t u, \t v)$ be two non-null non-negative bounded stationary solutions of the system \eqref{mf}. We will prove that they coincide. We define, for $\e>0$, 
$$
(u_\e,v_\e) := (u+\e\frac{\nu}{\mu},v+\e),
$$
and
$$
\theta_{\e}:=\max\{\theta>0 \ : \ (u_\e,v_\e) \geq \theta(\t u,\t v)\},
$$
which is positive.
Let us show that one of the following occurs:
$$
\text{Either}\ \ \exists x'_{\e} \in \mathbb{R},
\ \; u_\e(x'_{\e})=\theta_\e \t u(x_\e)
\quad\text{or}\quad
\exists(x_{\e},y_{\e}) \in \mathbb{R}\times[0,+\infty),
\ \;v_\e(x_{\e},y_{\e})=\theta_\e \t v(x_\e,y_\e).
$$
By definition of $\theta_\e$, for every $n\in \N$, we can find either
$x'_n$ such that $u_\e(x'_n)< (\theta_\e + \frac{1}{n})\t u(x'_n)$, or
$(x_n,y_n)$ such that $v_{\e}(x_n,y_n)<(\theta_\e+\frac{1}{n})\t v(x_n,y_n)$. For every $\e >0$, the norm of these points is bounded independently of $n\in\N$, because $\t u$ and $\t v$ converge to zero at infinity, by Lemma \ref{dec}. Therefore, 
because either $(x'_n)_n$ or $((x_n,y_n))_n$ 
has an infinite number of elements, 
we can define $(x'_\e)$ or $(x_\e,y_\e)$ as the limit of a subsequence of 
either $(x'_n)_n$ or $((x_n,y_n))_n$ respectively.

By definition, the positive reals $(\theta_{\e})_{\e>0}$ are increasing with respect to~$\e$. We define
$$
\theta_{0} := \lim_{\e \searrow 0^+}\downarrow \theta_{\e}.
$$
The rest of the proof is dedicated to show that $\theta_{0} \geq 1$. This will yield that $u\geq  \t u$ and $v \geq  \t v$. Exchanging the roles of $(u, v)$ and $(\t u, \t v)$ in what precedes, we then get $u = \t u$ and $v = \t v$, hence uniqueness.

We argue by contradiction, assuming that $\theta_{0}<1$. From now on, we assume that $\e$ is small enough so that $\theta_{\e}<1$. We proceed in two steps: we first derive
some estimates for the contact points $x'_{\e}$ or $(x_{\e},y_{\e})$, then we use them to get 
a~contraction.


\medskip
\emph{Step 1. Boundedness of the contact points.}\\
This step is dedicated to show that there is $R>0$ such that
\begin{equation}\label{estepsilon}
\forall \e >0, \quad \vert x'_{\e} \vert \leq R\ \text{ or }\ \vert (x_{\e} , y_{\e}) \vert \leq  R,
\end{equation}
i.e., $\vert x'_{\e} \vert$ or $\vert (x_{\e},y_{\e})\vert $ are bounded independently of $\e$. From the uniform regularity of $f$ together with \eqref{BFZ}, we infer that there is $\eta>0$ so that $v \mapsto f(x,y,v)$ is non-increasing in $ [0,\eta]$ if $\vert (x,y) \vert \geq R$.
Because $(u, v)$ is a stationary solution of \eqref{mf}, Lemma \ref{dec} implies that $v(x,y)$ goes to zero as $\vert (x,y) \vert$ goes to $+\infty$. Hence, up to decreasing $\e$ so that $\e <\eta$ and up to increasing $R$, we assume that
$$
v(x,y) \leq \eta -\e \quad \text{ for }\  \vert(x,y)\vert \geq R.
$$
Hence, if $\vert (x,y) \vert \geq R$, we see that
$$
-d\Delta v_\e-c\partial_{x}v_\e - f(x,y,v_\e) \geq f(x,y,v)-f(x,y,v_\e)\geq 0.
$$
Therefore,
\begin{equation*}
\left\{
\begin{array}{rll}
-D\partial_{xx}u_\e - c\partial_{x}u_\e&= \nu v_\e\vert_{y=0}-\mu u_\e, \quad &\text{ for }\  \vert x \vert > R, \\
-d\Delta v_\e -c\partial_{x}v_\e &\geq f(x,y,v_\e), \quad  &\text{ for }\   \vert (x,y) \vert  > R, \\
-d\partial_{y}v_\e \vert_{y=0} &= \mu u_\e-\nu v_\e\vert_{y=0}, \quad  &\text{ for }\  \vert x \vert > R.
\end{array}
\right.
\end{equation*}
Moreover, because we assumed that $\theta_{\e}<1$, we have $\theta_{\e}f(x,y,\t v) \leq f(x,y,\theta_{\e} \t v)$. Hence, $(\theta_\e \t u, \theta_\e \t v)$ is a stationary subsolution of \eqref{mf}. Because $(u_\e, v_\e) \geq (\theta_\e \t u, \theta_\e \t v)$, the elliptic strong comparison principle (see Remark \ref{req:cp})
yields that, if the point at which we have either 
$u_{\e}(x'_{\e}) = \theta_\e \t u(x'_\e)$ or $v_{\e}(x_{\e} , y_{\e}) = \theta_\e \t v(x_{\e} , y_{\e})$ satisfied
$$
\vert x'_{\e} \vert > R\ \text{ or }\ \vert (x_{\e} , y_{\e}) \vert >  R,
$$
then we would have
$$
(u_{\e},v_{\e}) \equiv  \theta_\e ( \t u, \t v).
$$
This is impossible because $u(x)\to\e \frac{\nu}{\mu}$ and 
$\t u(x)\to 0$ as $\vert x \vert$ goes to $+\infty$. We have reached a contradiction, showing that \eqref{estepsilon} holds true.

\medskip
\emph{Step 2. Taking the limit $\e \to 0$.}\\
The estimate \eqref{estepsilon} implies that, up to extraction of a suitable subsequence, 
either $x'_{\e}$ or $(x_{\e},y_{\e})$ converge as $\e$ goes to zero
to some limit $x_{0} \in \mathbb{R}$ or $(x_{0}, y_{0}) \in \R\times[0,+\infty)$.
Hence
\begin{equation*}
u \geq \theta_{0} \t u,\quad
v \geq \theta_{0} \t v,
\end{equation*}
and either $u(x_{0}) =\theta_{0} \t u(x_{0})$ or $v(x_{0},y_{0})=\theta_{0} \t v(x_{0},y_{0})$.  Because $(u, v)>(0,0)$, owing to the elliptic strong comparison principle (cf. Remark \ref{req:cp}), this yields $\theta_{0}>0$. As before, we can use the elliptic strong comparison principle for \eqref{mf} (see Remark \ref{req:cp}) with the solution $(u,v)$ and the subsolution $\theta_0(\t u, \t v)$ to find that these couples coincide everywhere, namely
%
$$
u \equiv \theta_{0} \t u\ \text{ and }\ v \equiv \theta_{0} \t v.
$$ 
Plotting the latter in \eqref{mf} we obtain
$$
\theta_{0}f(x,y,\t v) = f(x,y,\theta_{0}\t v) \quad \text{ for }\ (x,y)\in \mathbb{R}\times\mathbb{R}^{+}.
$$
Recalling that we assumed that $\theta_{0}<1$, we find a contradiction with the
hypothesis~\eqref{KPPCC}. 
This shows that $\theta_0\geq 1$, which concludes the proof.
\end{proof}

\subsubsection{The persistence/extinction dichotomy}\label{SectionExistence}

We proved in the previous section that there is at most one non-trivial bounded positive stationary solution of the semilinear problem \eqref{mf}. Building on that, we prove now Theorem \ref{persistence}.

\begin{proof}[Proof of Theorem \ref{persistence}.]
In the whole proof, $(u,v)$ denotes the solution of the parabolic problem \eqref{mf} arising from a non-negative not identically equal to zero compactly supported initial datum $(u_{0},v_{0})$. We prove separately the two statements of the Theorem.

\medskip
\emph{Statement $(i)$.} \\
Assume that $\lambda_{1}<0$. Owing to Proposition \ref{limit}, we can take $R>0$ large enough so that $\lambda_{1}^{R}<0$. Let $(\phi_{R},\psi_{R})$ be the corresponding
principal eigenfunction
provided by Proposition~\ref{limit}. Using the fact the $u(1,\cdot)>0$ and $v(1,\cdot,\cdot)>0$,
as a consequence of the parabolic comparison principle Proposition \ref{pro:cp},
and that $(\phi_{R},\psi_{R})$, extended by $(0,0)$ outside of its support, is compactly supported, we can  find $\e>0$ such that
\begin{equation*}
\e(\phi_{R},\psi_{R}) \leq (u(1,\cdot),v(1,\cdot,\cdot)).
\end{equation*}
%
%
Up to decreasing $\e$, the regularity hypotheses on $f$
combined with the fact that $\l^R<0$ implies that
$\e(\phi_{R},\psi_{R})$ (extended by $(0,0)$ outside its support) is 
a generalized stationary subsolution of \eqref{mf}. 
On the other hand, for $M>0$ sufficiently large, the pair 
$(\frac{\nu}{\mu}M,M)$ is a stationary supersolution of \eqref{mf}, due to 
hypothesis \eqref{sat}. 
Up to increasing $M$, we can assume that $(\frac{\nu}{\mu}M,M)>(u(1,\cdot),v(1,\cdot,\cdot))$.

As a standard application of the parabolic comparison principle, Proposition~\ref{pro:cp},
one sees that the solution of \eqref{mf} arising from $\e(\phi_{R},\psi_{R})$ (respectively from $(\frac{\nu}{\mu}M,M)$)  is time-increasing (respectively time-decreasing), and converges locally uniformly to a stationary solution, thanks to the parabolic estimates. Owing to the elliptic strong comparison principle (see Remark \ref{req:cp}),
this solution is positive. Proposition~\ref{Liouville} implies that these limiting solutions are actually equal, and by comparison~$(u,v)$ also converges to this positive stationary solution. This proves the statement $(i)$ of the theorem.

\medskip
\emph{Statement $(ii)$.} \\
Assume now that $\l \geq 0$. Let $(U,V)$ be a bounded non-negative stationary solution 
of~\eqref{mf}. We start to show that $(U,V) \equiv (0,0)$. 
We argue by contradiction, 
assuming that this stationary solution is not identically equal to zero. 
Let $(\phi,\psi)$ be a positive generalized principal eigenfunction associated with $\lambda_{1}$, provided by Proposition~\ref{limit}. The fact that $\lambda_{1}\geq0$, combined with the
Fisher-KPP hypothesis \eqref{KPPCC}, implies that $(\phi,\psi)$ is a supersolution of \eqref{mf}. For $\e>0$, we define
$$
\theta_{\e} := \max \left\{   \theta>0 \  : \ (\phi,\psi)+\left(\e\frac{\nu}{\mu},\e\right) \geq \theta(U,V)  \right\}.
$$
%
%
Hence, for $\e>0$, there is either $x'_{\e}\in \R$ or $(x_{\e},y_{\e})\in \R\times [0,+\infty)$ such that 
$$
\phi(x'_{\e}) + \e\frac{\nu}{\mu} = \theta_\e U(x'_{\e})\ \text{ or } \ \psi(x_{\e},y_{\e}) + \e = \theta_\e V(x_{\e},y_{\e}).
$$
Arguing as in the proof of Proposition \ref{Liouville}, Step 1, we find that the norm of the contact points $x'_{\e}$ or $(x_\e, y_\e)$ is bounded independently of $\e$. 
Because $\theta_{\e}$ is increasing with respect to $\e$, it converges to a limit $\theta_{0}\geq0$ as~$\e$ goes to $0$. Up to extraction, we have that either $x'_{\e}$ or $(x_{\e} , y_{\e})$ converges to some $x'_{0} \in \mathbb{R}$ or $(x_{0},y_{0}) \in \R\times[0,+\infty)$ as $\e$ goes to zero. Taking the limit $\e \to 0$ then yields
$$
(\phi,\psi)  \geq \theta_0(U,V),
$$
and either $\phi(x'_{0})  = \theta_0 U(x'_{0})$ or $\psi(x_{0},y_{0})  = \theta_0 V(x_{0},y_{0})$. In both cases, the elliptic strong comparison principle, cf.~Remark \ref{req:cp}, implies that
$$
(\phi,\psi)  \equiv \theta_0 (U,V),
$$
Owing to the hypothesis \eqref{KPPCC} on $f$, this is possible only if $\theta_0 = 0$, but this would contradict the strict positivity of $(\phi, \psi)$.
We have reached a contradiction: there are no non-negative non-null bounded stationary solutions of \eqref{mf} when $\lambda_{1}~\geq~0$.

We can now deduce that extinction occurs: for a given 
compactly supported initial datum $(u_{0},v_{0})$, we choose $M>0$ large enough so that the couple $(\frac{\nu}{\mu}M , M)$ is a stationary supersolution of \eqref{mf} and, in addition,
$$
\Big(\frac{\nu}{\mu}M , M\Big)\geq (u_{0},v_{0}).
$$
We let $(\ol u, \ol v)$ denote the solution of \eqref{mf} arising from the initial datum $(\frac{\nu}{\mu}M , M)$. It is time-decreasing and converges locally uniformly to a stationary solution. The only one being $(0,0)$, owing to Proposition~\ref{Liouville}, we infer that $(\ol u ,\ol v)$ converges to zero locally uniformly as $t$ goes to $+\infty$. Lemma \ref{dec} implies that the convergence is actually uniform. The same holds for $(u,v)$, thanks to the parabolic comparison principle, Proposition \ref{pro:cp}.
%
%
%
\end{proof}

We conclude this section by stating the dichotomy analogous to~Theorem \ref{persistence} 
in the case of the system ``without the road", \eqref{no road mf}.

\begin{prop}\label{persistence no road}
Let $\lambda_{N}$ be the generalized principal eigenvalue of the system \eqref{lsmfn}. 

\begin{enumerate}[(i)]

\item If $\lambda_{N}<0$, the system \eqref{no road mf} admits a unique positive bounded stationary solution and persistence occurs.

\item If $\lambda_{N}\geq0$, the system \eqref{no road mf} does not admit any positive stationary solution, and extinction occurs.

\end{enumerate}
\end{prop}
This result can be proved similarly to Theorem \ref{persistence}, or, alternatively, one can recall the results of \cite{BRforcedspeed}. In that paper, the authors consider the problem \eqref{OriginalCCmf} set on the whole plane, but their results adapt to the problem \eqref{no road mf} thanks to Lemma \ref{lemma lambda N}.
%

\section{Influence of a road on an ecological niche}\label{section niche}
In this section, we study the effect of a road on an ecological niche. In terms of our models, this means that we compare the system ``with the road" \eqref{mf} with the system ``without the road" \eqref{no road mf}, when $c=0$, i.e., when the niche does not move.

\subsection{Deleterious effect of the road on a population in an ecological niche}\label{road lethal}

This section is dedicated to the proof of Theorem \ref{th road lethal}, which answers Question \ref{question 1}. We start with a technical result.

\begin{prop}\label{road no help positivity}
Assume that $c=0$. Let $\lambda_{1}$ and $\lambda_{N}$ be the generalized principal eigenvalues of the model ``with the road" \eqref{lsmf} and of the model ``without the road"~\eqref{lsmfn} respectively. Then, 
$$
\lambda_{N} \geq 0 \implies \lambda_{1} \geq 0.
$$
\end{prop}
This proposition readily yields the statement $(i)$ of Theorem \ref{th road lethal}. Indeed, suppose that extinction occurs for the system ``without the road" \eqref{no road mf}. Then Proposition~\ref{persistence no road} implies that $\lambda_{N} \geq 0$ and thus Proposition \ref{road no help positivity} gives us that $\lambda_{1}\geq0$. Theorem \ref{persistence} then entails that extinction occurs for \eqref{mf}.

\begin{req}
In view of Proposition \ref{road no help positivity}, it might be tempting to think that $\l \geq \lambda_N$. However, this is not always the case. Indeed, taking $\psi=0$ in \eqref{eqvariational}, we find that $\l^R \leq \frac{\pi^2}{4DR^2}+\mu$, and therefore
$$\l = \lim_{R\to +\infty} \l^R \leq \mu.$$
Now, $\lambda_{N}$ does not depend on $\mu$, and for $\mu$ fixed can be made arbitrarily large. This can be achieved for instance by choosing $f = \rho f^L$ as in \eqref{f l} with $L$ sufficiently negative and $\rho$ sufficiently large.
\end{req}

\begin{proof}[Proof of Proposition \ref{road no help positivity}] Assume that $\lambda_{N}\geq 0$. 
Then $\lambda_{N}^{R} \geq 0$ for any $R>0$, thanks to~\eqref{cv eig no road}. 
Because $c=0$, on the one hand, the variational formula~\eqref{eqvariational no road} for $\lambda_{N}^{R}$ gives~us, for all $R>0$
\begin{equation*}
\forall \psi \in \t H^1_{0}(\O_R), \quad   \frac{  \int_{\Omega_{R}} \left( d\vert \nabla \psi \vert^{2} - m\psi^{2} \right)   }{ \int_{\Omega_{R}}\psi^{2} } \geq 0.
\end{equation*}H
On the other hand, \eqref{eqvariational} implies that
\begin{equation*}
\lambda_{1}^{R} \geq \inf_{(\phi,\psi) \in \mathcal{H}_{R}} \frac{  \int_{\Omega_{R}} \left( d\vert \nabla \psi \vert^{2} - m\psi^{2} \right)   }{  \int_{\Omega_{R}}\psi^{2} }\frac{ \nu\int_{\Omega_{R}}\psi^{2}}{\mu\int_{I_R}\phi^{2} + \nu\int_{\Omega_{R}}\psi^{2} }.
\end{equation*}
Gathering these inequalities, we get $\lambda_{1}^{R}\geq 0$. Because this is true for every $R>0$, taking the limit $R \to +\infty$ proves the result. 

\end{proof}

The proof of Theorem \ref{th road lethal} $(ii)$ is more involved. The key tool is the following.

\begin{prop}\label{prop road bad}
Assume that $c=0$ and that the parameters $d,\mu,\nu$ are fixed. For $L\in \R$ and $D>0$, 
let $\lambda_{N}(L)$ and $\l(L,D)$ denote the generalized principal eigenvalues of \eqref{lsmfn} and \eqref{lsmf} respectively, 
with nonlinearity $f^{L}$ given by \eqref{f l}. 
Then, for every $D>0$, there exists $L^{\star}\in \R$ such that
$$
\lambda_{N}(L^{\star}) <0< \lambda_{1}(D,L^{\star}).
$$

\end{prop}

\begin{proof} \emph{Step 1. Finding $L$ that yields $\lambda_{N}=0$.}\\
Let us first observe that
$$
\lim_{L\to -\infty} \lambda_{N}(L) > 0 > \lim_{L\to +\infty} \lambda_{N}(L).
$$
Indeed, owing to formula \eqref{eqvariational no road}, $L \mapsto \lambda_N(L)$ is 
non-increasing on $\R$, then it admits limits
as $L$ goes to $\pm\infty$. 
Moreover, Proposition \ref{prop lambda n} yields, for every $R>0$, 
$$
\lim_{L\to +\infty}\lambda_{N}(L) \leq 
\inf_{\psi \in \t H^1_{0}(\O_R)} \frac{ d \int_{\Omega_R} \vert \nabla \psi \vert^{2}}{\int_{\O_R}\psi^{2} } - 1.
$$
Then, taking the limit as $R\to+\infty$ and using the well-known fact that the quantity
$$
\inf_{\psi \in \t H^1_{0}(\O_R)} \frac{  \int_{\Omega_R} \vert \nabla \psi \vert^{2}}{\int_{\O_R}\psi^{2} }
$$
coincides with the principal eigenvalue of the Laplace operator on $B_R\subset
\R^2$ under Dirichlet boundary condition, which converges to $0$ as $R$ goes to $+\infty$, we find $\lim_{L\to +\infty} \lambda_{N}(L) < 0$.

Now, by definition of $f^L$, if $-L$ is large enough, we have that $f^L<-\frac{1}{2}$. Hence,~\eqref{eqvariational no road} implies that, for every such $L$ and for $R>0$,
$$
 \lambda_{N}^{R}(L) = \inf_{\psi \in \t H^1_{0}(\O_R)} \frac{  \int_{\Omega_{R}} \left( d\vert \nabla \psi \vert^{2} - m^L\psi^{2} \right)  }{\int_{\Omega_{R}}\psi^{2} }\geq \frac{1}{2},
$$
and then $\lim_{L \to -\infty}\lambda_{N}(L)>0$. Owing to Proposition \ref{prop lambda n}, $L \mapsto \lambda_{N}(L)$ is a continuous function on $\R$, and we can then define
$$
\ol L := \max\{L \in \R \ : \  \lambda_N(L) \geq 0 \}.
$$
It follows that $\lambda_N(\ol L) =0$ and $\lambda_N(L) <0$ for $L>\ol L$.


\medskip
\emph{Step 2. For every $D>0$, there holds $\l(\ol L,D)>0$.}\\
Assume by contradiction that there is $D>0$ such that $\l(\ol L,D)\leq 0$. Owing to the last statement of Proposition \ref{prop function}, for any $D'\in (0,D)$, we have
$$
0 \geq \l(\ol L, D) > \l(\ol L,D').
$$
This means that $\lambda_1(\ol L,D')<0 = \lambda_N(\ol L)$, contradicting Proposition \ref{road no help positivity}.

\medskip
\emph{Step 3. Conclusion.}\\
Let $D>0$ be given. As stated in Proposition \ref{prop function}, $L \mapsto \lambda_{1}(L,D)$ is a continuous function and thus, by Step $2$, $\l(L^\star,D)>0$ if $L^\star>\ol L$, with $L^\star$ sufficiently close to $\ol L$.
On the other hand, recalling the definition of $\ol L$, we have $\lambda_N(L^\star)<0$. This concludes the proof.
%
%
\end{proof}
Combining Proposition \ref{prop road bad} with Theorem \ref{persistence} and Proposition \ref{persistence no road} we derive the statement $(ii)$ of Theorem \ref{th road lethal}.

\subsection{Influence of the diffusions $D$ and $d$}\label{diffusions}

\subsubsection{Extinction occurs when $d$ is large}

This section is dedicated to the proof of Theorem \ref{d infinity}. A similar result is derived for the model ``without the road" in \cite{BRforcedspeed}. Throughout this section, we let $\lambda_{1}(d)$ be the generalized principal eigenvalue of \eqref{lsmf} with $c=0$, $D,\mu,\nu,f$ fixed and $d$ variable. We start with a technical proposition.
\begin{prop}\label{est d grand}
Let $\lambda_{1}(d)$ be the generalized principal eigenvalue of \eqref{lsmf} with $c=0$ and diffusion in the field $d>0$. Then,

$$
\liminf_{d\to +\infty}\lambda_{1}(d) \geq \min\Big\{ \mu , -\limsup_{\vert (x,y) \vert \to +\infty}m(x,y)\Big\} >0.
$$
\end{prop}
%

\begin{proof} The proof relies on the construction of suitable test functions for the formula~\eqref{fgpe}. It is divided into five steps. 

\medskip
\emph{Step 1. Defining the test function.} \\
We take $\lambda \geq 0$ such that
$$ \lambda < \min\{ \mu , -\limsup_{\vert (x,y) \vert \to +\infty}m(x,y)\}.
$$
By hypothesis \eqref{BFZ}, we can find $M$ large enough so that
$$
K:=-\sup_{\vert (x,y) \vert \geq M}  m(x,y) >0
$$
whence $\lambda \in (0,\min\{ \mu , K\})$. We take $A \in (\lambda,\mu)$ and we define
\begin{equation}\label{v u d}
\begin{cases}
\Psi(x,y) := \phi(x) + \psi(y), \\
\Phi(x) := \frac{\nu}{\mu - A}\Psi(x,0),
\end{cases}
\end{equation}
where
\begin{equation}
\begin{cases}
\phi(x) := \cos(\alpha_{d} x) 1_{\vert x \vert \leq M} + \cos(\alpha_{d} M) e^{\frac{1}{d^{\beta}}(M-\vert x \vert)}1_{\vert x \vert > M}\quad &\text{ for } \ x \in \mathbb{R}, \\
\psi(y) := e ^{-\frac{y}{d^{\beta}}}\quad &\text{ for } \ y \in \mathbb{R},
\end{cases}
\end{equation}
with $\beta \in (\frac{1}{2},1)$ and $\alpha_{d}$ being the unique real number such that
$$
\alpha_{d} \in \left(0,  \frac{\pi}{2M}\right) \ \text{ and } \ \tan(\alpha_{d} M) = \frac{1}{d^{\beta}\alpha_{d}}.
$$
Then, the functions $\Phi$ and $\Psi$ given by \eqref{v u d} are non-negative and belong to $W^{2,\infty}(\R)$ and $W^{2,\infty}(\R \times [0,+\infty))$ respectively. They are suitable test functions for formula~\eqref{fgpe}.

The next steps are dedicated to proving the following inequalities
\begin{equation}\label{sur}
\begin{cases}
D\Phi^{\prime \prime} - \mu \Phi + \nu \Psi\vert_{y=0} +\lambda \Phi &\leq 0, \\
d \Delta \Psi +m\Psi + \lambda \Psi &\leq 0, \\
d\partial_{y}\Psi\vert_{y=0}-\nu \Psi\vert_{y=0} + \mu \Phi &\leq 0.
\end{cases}
\end{equation}
Observe that $\Psi\leq 2$, $\Phi\leq \frac{2\nu}{\mu-A}$. These inequalities will be used several times in the following computations.

\medskip
\emph{Step 2. The boundary condition.} \\
Let us first check that $(\Phi,\Psi)$ satisfies the third inequality of \eqref{sur}. We have
\begin{equation*}
\begin{array}{lcc}
d\partial_{y}\Psi\vert_{y=0}-\nu \Psi\vert_{y=0} +\mu \Phi &=& d \psi^{\prime}(0)+A\Phi \\
  &\leq& -d^{1-\beta} +\frac{2 A\nu}{\mu - A} .
\end{array}
\end{equation*}
Because $1-\beta >0 $, this is negative is $d$ if large enough. 

\medskip
\emph{Step 3. Equation for $\Phi$.}\\
Let us check that the first inequality of \eqref{sur} holds true almost everywhere. First, assume that $\vert x \vert < M$. Then
\begin{equation*}
\begin{array}{lll}
D\Phi^{\prime\prime}-\mu \Phi+\nu \Psi\vert_{y = 0} + \lambda \Phi &=& D\Phi^{\prime \prime} -A \Phi + \lambda \Phi \\
 &=& \frac{\nu}{\mu-A}(-D\alpha_{d}^{2}\cos(\alpha_{d} x) +(\lambda-A)(\cos(\alpha_{d} x) + 1 )).
\end{array}
\end{equation*}
Because $\lambda<A$, this is negative. Now, if $\vert x \vert > M$, we have

\begin{equation*}
\begin{array}{lll}
D\Phi^{\prime\prime}-\mu \Phi+\nu \Psi\vert_{y=0} + \lambda \Phi &=& D\Phi^{\prime \prime} -A \Phi + \lambda \Phi \\
 &=& \frac{\nu}{\mu-A} \left( (\frac{D}{d^{2\beta}} +\lambda-A)\cos(\alpha_{d} M)e^{\frac{1}{d^{\beta}}(M-\vert x \vert)} +\lambda-A \right).
\end{array}
\end{equation*}
Because $\lambda< A $, this is negative if $d$ is large enough.

\medskip
\emph{Step 4. Equation for $\Psi$.} \\
Finally, let us check that the second inequality of \eqref{sur} holds almost everywhere. First, when $\vert x \vert < M$, we have (a.e.)
\begin{equation*}
\begin{array}{lll}
d \Delta \Psi +m \Psi+\lambda \Psi &=& d(\phi^{\prime \prime}+\psi^{\prime \prime})+(\lambda+m)(\phi+\psi)\\
 &\leq& -d\alpha_{d}^{2}\cos(\alpha_{d} M)+d^{1-2\beta} +2( \sup m +\lambda). 
\end{array}
\end{equation*}
Observe that 
$$
\alpha^{2}_{d} \sim \frac{1}{Md^{\beta}} \quad \text{ as }\  d \ \text{ goes }\ to +\infty.
$$
 Therefore, 
 $$
 -d\alpha_{d}^{2}\cos(\alpha_{d} M)\sim- \frac{1}{M}d^{1-\beta} \quad \text{ as }\  d \ \text{ goes }\ to +\infty.
 $$
Because $\beta<1$, this goes to $-\infty$ as $d$ goes to $+\infty$. On the other hand, $d^{1-2\beta}$ goes to zero as $d$ goes to $+\infty$, because $1-2\beta<0$. Then, for $d$ large enough,
$$
-d\alpha_{d}^{2}\cos(\alpha_{d} M)+d^{1-2\beta} +2( \sup m +\lambda)
$$
is negative.

If $\vert x \vert > M$, we have
\begin{equation*}
\begin{array}{lll}
d \Delta \Psi +m \Psi+\lambda \Psi &=& d(\phi^{\prime \prime}+\psi^{\prime \prime})+(\lambda+m )(\phi+\psi)\\
 &\leq& d(\frac{1}{d^{2\beta}}\cos(\alpha_{d} M)e^{\frac{1}{d^{\beta}}(M-\vert x \vert)} +d^{-2\beta} e^{-\frac{1}{d^{\beta}} y} )+(-K+\lambda)(\phi+\psi) \\
 &\leq& (d^{1-2\beta}+\lambda-K)\phi+(d^{1-2\beta}+\lambda-K)\psi.
\end{array}
\end{equation*}
Because $\lambda<K$ and $\beta > \frac{1}{2}$, this is negative for $d$ large enough.

\medskip
\emph{Step 5. Conclusion.} \\
Gathering all that precedes, we have shown that, for $d$ large enough, \eqref{sur} is verified. Owing to the formula \eqref{fgpe} defining $\l(d)$, this implies that $\lambda_1(d) \geq \lambda$, for $d$ large enough. The fact that we choose $\lambda$ arbitrarily yields the result.
\end{proof}

We are now in a position to prove Theorem \ref{d infinity}.
\begin{proof}[Proof of Theorem \ref{d infinity}]
Owing to Proposition \ref{est d grand}, we see that there is $\overline{d}>0$ such that
$$
\forall d > \ol d, \quad \l(d)\geq 0.
$$
It is readily seen from the variational formula \eqref{eqvariational} that the function $d \in \R^+ \mapsto \lambda_{1}(d)$ is non-decreasing. We can define
$$
d^\star := \min\{ d \geq 0 \ : \ \l(d) \geq 0\}.
$$
Theorem \ref{persistence} then yields the result.
\end{proof}

We have proved that extinction occurs when the diffusion in the field is above a certain threshold. It is natural to wonder whether the same result holds in what concerns the diffusion on the road: is there $D^{\star}$ such that extinction occurs for \eqref{mf} with $c=0$ when $D\geq D^{\star}$? Without further assumptions on the coefficients, the answer is no in general, as shown in the following section.

\subsubsection{Influence of the diffusion on the road}\label{section Robin}
This section is dedicated to proving the following statement.
\begin{prop}\label{D no influence}
Consider the system \eqref{mf}, with $c=0$ and
$f=f^L$ given by \eqref{f l}.
For $L$ large enough (depending on $d$), persistence occurs for every $D, \mu, \nu>0$.
\end{prop}
This result is the counterpart of Theorem \ref{d infinity}, which 
asserts that increasing the diffusion $d$ in the field, the system is inevitably led to extinction
(under assumption~\eqref{BFZ}).
Proposition~\ref{D no influence} shows that this is not always the case if, instead of $d$, one increases
the diffusion on the road. 
It is also interesting to compare it with Theorem~\ref{th road lethal}. While the latter states that the road always has a deleterious influence on the population, Proposition \ref{D no influence} means that this effect is nevertheless limited. Indeed, if the favorable zone is sufficiently large, then, no matter what happens on the road, that is, regardless of the coefficients $D,\mu,\nu$, there will always be persistence.

\begin{proof}[Proof of Proposition \ref{D no influence}]
For $R>0$, let $\lambda_R$ and $\phi_R$ denote the principal eigenvalue and (positive)
eigenfunction of 
$-\Delta$ on $B_R\subset \R^2$, under Dirichlet boundary condition.
We take $R$ large enough so that $\lambda_R<\frac1d$ (it is well known that
$\lambda_R\searrow0$ as $R$ goes to $+\infty$).
Then define $ \ul v(x,y) := \phi_R(x,y-2R)$ for $(x,y) \in 
\overline{B}_{R}(0,2R)$. The definition of~$f^L$ and the fact that $d\lambda_R<1$
allows us to find $L$ sufficiently large so that 
$$\min_{(x,y)\in\ol B_{3R}}\,m^{L}(x,y,0) >d\lambda_R.$$
As a consequence,
$$-d \Delta \ul v =d\lambda_R\ul v<m^{L}(x,y,0)\ul v\quad\ \text{ in }\overline{B}_{R}(0,2R).$$
Owing to the regularity of $f$, we can take $\e>0$ small enough so that $\e \ul v$ satisfies
$-d \Delta(\e \ul v)<f^{L}(x,y,\e\ul v)$  in $\overline{B}_{R}(0,2R)$.
The parabolic comparison principle 
in the ball ${B}_{R}(0,2R)$ implies that the solution of \eqref{mf} with $c=0$,
arising from the initial 
datum $(0,\e \ul v)$ with $\e\ul v$ extended by $0$ outside 
$\overline{B}_{R}(0,2R)$, is larger than or equal to $(0,\e \ul v)$ for all positive times.
In particular, extinction does not occur and hence, by Theorem~\ref{persistence}, 
we necessarily have persistence. Because this is true independently of the values of $D,\mu,\nu$,
the proof is complete.
%
\end{proof}
It is worthwhile to note a few observations about this result. To prove Proposition \ref{D no influence}, we compared 
system~\eqref{mf} with the single equation in a ball, under Dirichlet boundary condition,
for which we are able to show that persistence occurs provided $L$ is sufficiently large.
The population dynamics intuition behind this argument is clear: the Dirichlet condition means that the individuals touching the boundary are ``killed"; it is therefore harder for the population to persist. 
One could have compared our system with the system with Dirichlet condition 
on the road instead, by showing that the generalized principal eigenvalue of the latter 
is always larger than $\l$. 
As a matter of fact, it is also possible to compare system~\eqref{mf} with the 
system with Robin boundary condition:
$$
\left\{
\begin{array}{rll}
\partial_t v -d \Delta v  &=f(x,y,v), &\quad  t>0, \ (x,y) \in \mathbb{R}\times\mathbb{R}^{+}, \\
-d \partial_{y}v\vert_{y=0} +\nu v\vert_{y=0} &=0 &\quad t>0,\ x \in \mathbb{R}\times\{0\}.
\end{array}
\right.
$$
This system describes the situation where the individuals can enter the road, but cannot leave it. It can actually be shown that the generalized principal eigenvalue of the linearization of this system, which we denote by $\lambda_{Robin}$, is larger than $\lambda_1$. We conjecture that
$$
\lambda_{1} \underset{D \to +\infty}{\longrightarrow} \min\{ \lambda_{Robin}, \mu\}.
$$
This is based on the intuition that,
as $D$ becomes large, the population on the road diffuses ``very fast'' and 
then is sent ``very far" into the unfavorable zone, where it dies.

\section{Influence of a road on a population facing a climate change}\label{section road}

\subsection{Influence of the speed $c$}\label{section c}

This section is dedicated to proving Theorem \ref{critical}. In this whole section, we assume that the coefficients $D,d,\mu,\nu$ are fixed in \eqref{lsmf}, and we let $\lambda_{1}(c)$ denote the generalized principal eigenvalue of \eqref{lsmf} and $\lambda_{1}^{R}(c)$ the principal eigenvalue of \eqref{eigborn}, as functions of the parameter $c\geq 0$.
We start with proving two preliminary results.

\begin{prop}\label{EstEig}
Let $\lambda_{1}(c)$ be the generalized principal eigenvalue of \eqref{lsmf}. Then
$$
\lambda_{1}(c) \geq \frac{1}{4} \min\left\{\frac{1}{d}, \frac{1}{D}\right\}c^{2}  - [m]^+.
$$
\end{prop}

\begin{proof}
Let $c\geq 0$ be chosen. For $R>0$, let $(\phi_{R},\psi_{R})$ denote the principal eigenfunction of \eqref{eigborn} and $\l^R$ the associated principal eigenvalue. Take $\kappa \in \mathbb{R}$. The idea is to multiply the system \eqref{eigborn} by the weight $x \mapsto e^{\kappa x}$, and to integrate by parts. At the end, optimizing over $\kappa$ will yield the result. We define 
$$
I_{\psi}:=\int_{\Omega_{R}}\psi_R(x,y) e^{\kappa x}dxdy \ \text{ and } \ I_{\phi}:=\int_{I_R}\phi_R(x) e^{\kappa x}dx.
$$ 
We multiply the equation for $\psi_R$ in \eqref{eigborn} by $e^{\kappa x}$ and integrate over 
$\Omega_{R}$ to get
\begin{equation}\label{croiss}
-d\int_{\Omega_{R}}  (\Delta \psi_R) e^{\kappa x} -c \int_{\Omega_{R}}  (\partial_{x}\psi_R) e^{\kappa x} =\int_{\Omega_{R}}  m(x,y)\psi_R e^{\kappa x} +\lambda^{R}_{1}I_{\psi}.
\end{equation}
We let $e_{x}$ denote the unit vector in the direction of the road, i.e., $e_{x} := (1,0)$, and $\nu$ the exterior normal vector to $\Omega_{R}$. We have
$$
\int_{\Omega_{R}}(\partial_{x}\psi_R)e^{\kappa x} =-\kappa \int_{\Omega_{R}}\psi_R e^{\kappa x} + \int_{\Omega_{R}}\partial_x(\psi_R e^{\kappa x  }).
$$
Because $\psi_R=0$ on $\partial \O_R\setminus I_R$, the Fubini theorem implies that $\int_{\Omega_{R}}\partial_x(\psi_R e^{\kappa x  })=0$.
Hence
$$
\int_{\Omega_{R}}(\partial_{x}\psi_R) e^{\kappa x} =-\kappa I_{\psi} .
$$
Using the divergence theorem, as well as the above equivalence, we  find that
\begin{align*}
-d\int_{\Omega_{R}}( \Delta \psi_R )e^{\kappa x}  &=  - d\int_{\partial \Omega_{R}} 
(\partial_{\nu}\psi_R) e^{\kappa x}+
d\kappa \int_{\Omega_{R}}(\partial_{x} \psi_R) e^{\kappa x} \\
&= -d\kappa^{2} I_{\psi} - d\int_{\partial \Omega_{R}\backslash I_R} (\partial_{\nu}\psi_R )
e^{\kappa x} + d\int_{I_R} (\partial_{y}\psi_R) e^{\kappa x} \\
&\geq -d\kappa^{2} I_{\psi} + d\int_{I_R} (\partial_{y}\psi_R )e^{\kappa x},
\end{align*}
where the last inequality comes from the fact that we have $\psi_R=0$ on $ \partial \Omega_{R}\backslash I_R$ and $\psi_R\geq0$ elsewhere, hence $\partial_{\nu}\psi_R\leq 0$ on $\partial \Omega_{R}\backslash I_R$. Then, \eqref{croiss} yields
\begin{equation}\label{croiss2}
0\leq \left( [\sup m]^++\lambda_{1}^{R} -\kappa c+d\kappa^{2} \right)I_{\psi} +\int_{I_R}(-d\partial_{y} \psi_R)e^{\kappa x}.
\end{equation}
Now, the boundary condition in \eqref{eigborn} combined with the equation satisfied by $\phi_R$ gives us
\begin{align*}
\int_{I_R}(-d\partial_{y} \psi_R)e^{\kappa x} &=\int_{I_R}(D \partial_{xx}\phi_R + c \partial_{x}\phi_R+\lambda_{1}^{R} \phi_R)e^{\kappa x} \\
  &= \int_{I_R}(D \partial_{xx}\phi_R+ c \partial_{x}\phi_R)e^{\kappa x}+\lambda_{1}^{R} I_{\phi}.
\end{align*}
Integrating by parts and arguing as before, we obtain
$$
\int_{I_R}(-d\partial_{y} \psi_R)e^{\kappa x} \leq (D\kappa^{2}-c\kappa+\lambda_{1}^{R})I_{\phi} .
$$
Then, \eqref{croiss2} implies that

\begin{equation*}
0\leq \left( d\kappa^{2}-c\kappa +[m]^++\lambda_{1}^{R}\right)I_{\psi} + \left(D\kappa^{2}-c\kappa+\lambda_{1}^{R}\right)I_{\phi}.
\end{equation*}
We write $\kappa := \alpha c$, using $\alpha \in \mathbb{R}$ as the new optimization 
parameter.
Because $I_{\phi}$ and $I_{\psi}$ are positive, we deduce that one of the following inequalities
necessarily holds:
$$
(d\alpha^{2}-\alpha)c^{2} +[\sup m]^++\lambda_{1}^{R}\geq0,\qquad  
(D\alpha^{2}-\alpha)c^{2}+\lambda_{1}^{R}\geq0.
$$
Namely, we derive
\begin{align*}
\lambda_{1}^{R}(c) &\geq \sup_{\alpha \in \mathbb{R}}
\Big(\min \{ -(d\alpha^{2}-\alpha)c^{2} -[\sup m]^+ , -(D\alpha^{2}-\alpha)c^{2}\}\Big) \\
 &\geq \sup_{\alpha \in \mathbb{R}}
 \Big(\min \{\alpha -d\alpha^{2},\alpha  -D\alpha^{2}\}c^{2}  - [\sup m]^+\Big) \\
  &\geq\frac{1}{4} \min\left\{\frac{1}{d}, \frac{1}{D}\right\}c^{2}  - [\sup m]^+ .
\end{align*}
Letting $R$ go to $+\infty$, we get the result.
\end{proof}

Proposition \ref{EstEig} implies that $\lambda_1(c)\geq0$ if
$$c \geq 2\sqrt{\max\{ d,D  \} [\sup m]^+}.$$ 
Owing to the continuity of $c\mapsto \lambda_1(c)$ 
(recalled in Proposition \ref{prop function} above), 
this allows us to define
$$c_{\star} := \min\{ c\geq0 \ : \ \lambda_1(c)\geq 0\},
	\qquad
	c^{\star} := \sup\{ c\geq0 \ : \ \lambda_1(c)< 0\},$$
with the convention that~$c^\star=0$ if the set in its definition is empty.
Moreover, $c_{\star}>0$ if and only if $\lambda_1(0)<0$. 
Thanks to Theorem~\ref{persistence}, we have thereby proved Theorem~\ref{critical}.
%
%
%
%

As we mentioned in the introduction, Section \ref{Our model}, we actually conjecture that $c_{\star} = c^{\star}$. We prove that this is true when $d=D$.
%
\begin{prop}\label{d equal  D}
Assume that $d=D$ in \eqref{mf}. Then
$$
c_{\star} = c^{\star} = 2\sqrt{-d[-\lambda_{1}(0)]^{+}}.
$$
\end{prop}
%
This proposition is readily derived using the change of functions
$$
\tilde{\phi} := \phi e^{\frac{c}{2d}x}, \quad \tilde{\psi} := \psi e^{\frac{c}{2d}x}
$$
in \eqref{fgpe} to get
$$
\lambda_{1}(c) = \frac{c^{2}}{4d}+\lambda_{1}(0).
$$
%
%
%

We conclude this section by showing that $c\mapsto\l (c)$ attains its minimum at $c=0$.
This has a natural interpretation: it means that a population is more likely to persist if the favorable zone is not moving; in other words, the climate change always has a deleterious effect on the population, at least for what concerns~survival.

\begin{prop}\label{minzero}
Let $c\geq 0$. Then
$$
\lambda_{1}(c) \geq \lambda_{1}(0).
$$
\end{prop}

\begin{proof}
Take $R\geq 0$ and let $(\phi_{R},\psi_{R})$ be the principal eigenfunction of \eqref{eigborn} and $\l^R$ be the associated eigenvalue. We multiply the equation for $\psi_{R}$ in \eqref{eigborn}
 by $\psi_R$ and integrate over $\Omega_{R}$ to get
$$
\int_{\Omega_{R}}d\vert \nabla \psi_{R} \vert^{2} -d\int_{\partial \Omega_{R}} \psi_{R}\partial_{\nu}\psi_{R}-\int_{\Omega_{R}}m\psi_{R}^{2} = \lambda_{1}^{R}(c)\int_{\Omega_{R}}\psi_{R}^2,
$$
where we have used the fact that
$$
\int_{\Omega_{R}}c\psi_R\partial_x\psi_R=\frac c2\int_{\Omega_{R}}\partial_x(\psi_R)^2=0.
$$
Then, the boundary condition yields
\begin{equation}\label{ener1}
\int_{\Omega_{R}}d\vert \nabla \psi_{R} \vert^{2} -\int_{I_R} \psi_{R}\vert_{y=0}(\mu \phi_{R} - \nu \psi_{R}\vert_{y=0})-\int_{\Omega_{R}}m\psi_{R}^{2} = \lambda_{1}^{R}(c)\int_{\Omega_{R}}\psi_{R}^{2}. 
\end{equation}
Likewise, multiplying the equation on the road by $\phi_{R}$ and integrating, we have
\begin{equation}\label{ener2}
\int_{I_R}D\vert \phi_{R}^{\prime} \vert^{2} +\int_{I_R}\phi_{R}(\mu \phi_{R} - \nu \psi_{R}\vert_{y=0}) = \lambda_{1}^{R}(c)\int_{I_R}\phi_{R}^{2}. 
\end{equation}
Multiplying \eqref{ener1} by $\nu$ and \eqref{ener2} by $\mu$ and summing the two resulting equations yields
$$
\lambda_{1}^{R}(c) = \frac{  \mu\int_{I_R}D \vert \phi_{R}^{\prime} \vert^{2}  + \nu\int_{\Omega_{R}} \left( d\vert \nabla \psi_{R} \vert^{2} - m\psi_{R}^{2} \right) + \int_{I_R}(\mu \phi_{R}- \nu \psi_{R}\vert_{y=0})^{2}  }{\mu\int_{I_R}\phi_{R}^{2} + \nu\int_{\Omega_{R}}\psi_{R}^{2} }.
$$
Owing to Proposition \ref{variational}, this is greater than $\l^R(0)$, hence the result.
\end{proof}

\subsection{Positive effect of the road in keeping pace with a climate change}\label{section road helps}

In this section, we prove Theorem \ref{effect road}, whose Corollary \ref{cor effect road} answers Question \ref{question 2}.
The key observation is that, when $L$ goes to $+\infty$, the nonlinearity $f^{L}$ converges to $f^{\infty}(v) := v(1-v)$, the favorable zone then being the whole space, and the system~\eqref{mf} becomes, at least formally
\begin{equation}\label{fr}
\left\{
\begin{array}{rlll}
\partial_{t}u-D\partial_{xx}u -c\partial_{x}u &= \nu v\vert_{y=0}-\mu u , &\quad t >0,\ x\in \mathbb{R}, \\
\partial_{t}v-d\Delta v -c\partial_{x}v &= v(1-v), &\quad t >0 ,\ (x,y)\in \mathbb{R}\times \mathbb{R}^{+}, \\
-d\partial_{y}v\vert_{y=0}&= \mu u-\nu v\vert_{y=0}, &\quad t >0 ,\ x\in \mathbb{R}.
\end{array}
\right.
\end{equation}
This system is the road-field model \eqref{OriginalFR} from \cite{BRR1}, recalled in Section \ref{previous models}, rewritten in the frame moving in the direction of the road with speed $c \in \R$. The results of \cite{BRR1}, summarized here in Proposition \ref{rappel FR}, are obtained by constructing explicit supersolutions and subsolutions, and do not use principal eigenvalues. We need to rephrase Proposition \ref{rappel FR} in terms of the generalized principal eigenvalue. Namely, we define

\small
\begin{equation}\label{lambda BRR}
\begin{array}{ll}
\lambda_{H} :=\sup\Big\{  &\lambda\in\mathbb{R} \ : \ \exists (\phi,\psi) \geq 0,\ (\phi,\psi) \not\equiv (0,0) \ \text{ such that } \ \\
&D\phi^{\prime\prime}+c\phi^{\prime}-\mu \phi +\nu \psi\vert_{y=0}+\lambda \phi \leq 0 \ \text{on} \ \R, \\
 &d \Delta \psi +c\partial_{x}\psi +\psi + \lambda \psi  \leq 0 \ \text{ on }\ \R\times \R^+ ,\  d\partial_{\nu}\psi-\nu \psi\vert_{y=0} +\mu \phi \leq 0\ \text{on}\ \R \Big\}.
\end{array}
\end{equation}
\normalsize
Then, $\lambda_{H}$ is the generalized principal eigenvalue of
\begin{equation}\label{FR lin}
\left\{
\begin{array}{rlll}
-D\partial_{xx}\phi - c\partial_{x}\phi &= \nu \psi\vert_{y=0}-\mu \phi ,\quad &x\in \mathbb{R}, \\
-d\Delta \psi - c\partial_{x}\psi &=\psi,\quad &(x,y)\in \mathbb{R}\times \mathbb{R}^{+}, \\
-d\partial_{y}\psi\vert_{y=0}&= \mu \phi-\nu \psi\vert_{y=0}, \quad&x\in \mathbb{R}.
\end{array}
\right.
\end{equation}
This is the linearization at $(u,v) = (0,0)$ of the stationary system associated with the road-field model \eqref{fr}, in the frame moving in the direction of the road with speed~$c$. For $R>0$, we let $\lambda_H^R$ denote the principal eigenvalue of \eqref{FR lin} on the truncated domains $\O_R$ in the field and $I_R$ on the road. We know from \cite{BDR1} that $\lambda_H^R \to \lambda_H$ as $R$ goes to $+\infty$ and that there is a positive generalized principal eigenfunction associated with $\lambda_H$.

Consistently with our previous notations, we let $\lambda_{H}(c)$ denote the generalized principal eigenvalue of \eqref{FR lin} with $D,d,\mu,\nu>0$ fixed and with $c\in \R$ variable.

\begin{lemma}\label{prop comparaison fr}

Let $\lambda_{H}(c)$ denote the generalized principal eigenvalue of \eqref{FR lin}. Then

\begin{equation*}
\lambda_{H}(c)<0 \ \text{ for }\  c\in [0,c_{H}) \ \text{ and } \ \lambda_{H}(c) \geq 0 \ \text{ for }\ c\geq c_{H}.
\end{equation*}

\end{lemma}
\begin{proof}
We argue by contradiction. Assume that there is $c \in [0,c_{H})$ such that $\lambda_{H}(c) \geq~0$. Let $(\phi, \psi)\geq(0,0)$ be a generalized principal eigenfunction associated with $\lambda_{H}(c)$. Then $(\phi,\psi)$ is a stationary supersolution of \eqref{fr}, owing to the Fisher-KPP property, and because $\lambda_H(c)\geq 0$. We normalize it so that $\psi(0,0) =\frac{1}{2}$.

Now, let $(u_0,v_0)$ be a non-negative, not identically equal to zero compactly supported initial datum such that
$$
(u_0 , v_0) \leq (\phi,\psi).
$$ 
Let $(u,v)$ be the solution of \eqref{fr} arising from $(u_0 ,v_0)$. The parabolic comparison principle Proposition \ref{pro:cp} implies that
\begin{equation}\label{contra}
(u,v) \leq (\phi,\psi).
\end{equation}
However, because $0\leq c<c_{H}$, the main result of \cite{BRR1}, Proposition \ref{rappel FR} above, when translated in the moving frame of \eqref{fr}, yields
$$
(u(t,x) , v(t,x,y)) \underset{t\to +\infty}{\longrightarrow} \left(\frac{\nu}{\mu} , 1 \right),
$$ 
locally uniformly in $x$ and $(x,y)$. This contradicts \eqref{contra} because $\psi(0,0)= \frac{1}{2}$. Therefore, $\lambda_{H}(c) <0$ when $c\in~[0,c_{H})$.

Now, take $c > c_{H}$. If we had $\lambda_{H}(c) <0$, we could argue as in the proof of Theorem~\ref{persistence} to show that persistence occurs. However, in view of Proposition~\ref{rappel FR}, this cannot be the case. Hence, $\lambda_{H}(c)~\geq~0$. Because $\lambda_{H}$ is a continuous function of $c$ (see \cite[Proposition 2.4]{BDR1}), it must be the case that $\lambda_{H}(c_{H})=0$. This concludes the proof.
%
\end{proof}
The next proposition states that, in some sense, the system \eqref{mf} converges to the homogeneous system \eqref{fr} as $L$ goes to $+\infty$. In agreement with our previous notations, we let $\lambda_{1}(c,L)$ denote the generalized principal eigenvalue of \eqref{lsmf} with parameters $d,D>0$ fixed and with $c\in \R$ and nonlinearity $f^L$ given by \eqref{f l} variable.
\begin{prop}\label{prop f l}

Let $\lambda_{1}(c,L)$ be the generalized principal eigenvalue of \eqref{lsmf} with nonlinearity $f^{L}$ defined in \eqref{f l}.
Then
$$
\lambda_{1}(c,L) \underset{L \to +\infty}{\longrightarrow} \lambda_{H}(c)\quad \text{locally uniformly in } c.
$$
\end{prop}

\begin{proof}

First, because $m^{L}(\cdot,\cdot,0)\leq  1$, formulae \eqref{fgpe} and \eqref{lambda BRR} yield
$$
\forall  L >0, \  c \geq 0, \quad  \lambda_{H}(c) \leq \lambda_{1}(c,L),
$$
hence
$$
 \lambda_H(c) \leq \liminf_{L\to+\infty}\lambda_1(c,L).
$$
%
%
%
%
Let $\e\in(0,1)$ be fixed. In view of the definition of $f^L$ \eqref{f l}, for any $R>0$, we can take $L_R>0$ such that, for $L\geq  L_R$, $m^L(x,y,0) \geq 1 -\e$ on $\O_R$. Therefore,
\begin{equation}\label{lim l}
\forall R>0, \ L \geq L_R, \quad \lambda_1^{R}(c, L) \leq \lambda_{H}^{R}(c) + \e.
\end{equation}
Because $\lambda_{1}^{R}(c, L) \geq \lambda_{1}(c,L)$ and by arbitrariness of $\e$ in \eqref{lim l}, we find that
$$
\limsup_{L\to+\infty}\lambda_1(c,L)\leq \lambda_H(c),
$$
Hence,
$$
\lambda_{1}(c,L) \underset{L \to +\infty}{\longrightarrow} \lambda_{H}(c).
$$
This convergence is locally uniform with respect to~$c \geq~0$. Indeed, the continuity of the functions $c \mapsto \lambda_{1}(c,L)$ and $c \mapsto \lambda_{H}(c)$ combined with the fact that the family  ($\lambda_{1}(c,L))_{L >0}$ is decreasing and converges pointwise to $\lambda_{H}(c)$ as $L$ goes to $+\infty$ allows us to apply Dini's theorem.
\end{proof}
We are now in a position to prove Theorem \ref{effect road}.
\begin{proof}[Proof of Theorem \ref{effect road}]
As explained in the proof of Proposition \ref{prop f l} above, we have
$$
\forall  L\in \R, \ c\geq 0,\quad \lambda_H(c)\leq\lambda_1(c,L).
$$
Owing to Lemma \ref{prop comparaison fr}, we have $\lambda_H(c)\geq0$ for every $c \geq c_H$. By definition of $c^{\star}$, we find that
%
\begin{equation}\label{cv c 1}
\forall  L\in \R,\quad c^{\star} \leq c_H.
\end{equation}
Take $\eta\in(0,c_H)$. Lemma \ref{prop comparaison fr} yields $\lambda_{H}(c)<0$ for every $c \in [0,c_{H}-\eta]$. Because $\lambda_{1}(c,L)$ converges locally uniformly to $\lambda_{H}(c)$ as $L$ goes to $+\infty$, by Proposition~\ref{prop f l}, we find that there is $L^{\star}$ such that
$$
\forall L \geq L^{\star}, \ c\in [0,c_{H}-\eta], \quad \lambda_{1}(c,L)<0.
$$
Theorem \ref{persistence} then implies that persistence occurs in \eqref{mf} if $c\in[0,c_H-\eta]$ and $L\geq L^{\star}$. It follows from the definition of $c_{\star}$ that
\begin{equation}\label{cv c 2}
\forall L \geq L^{\star}, \quad c_{\star}\geq c_H-\eta.
\end{equation}
We can take $\eta$ arbitrarily close to zero, up to increasing $L^{\star}$ if need be. Combining \eqref{cv c 1} and~\eqref{cv c 2} yields the result.
%
%
\end{proof}

We can now deduce Corollary \ref{cor effect road} from Theorem \ref{effect road}
\begin{proof}[Proof of Corollary \ref{cor effect road}]
Assume that $D>2d$.
Consider the system \eqref{mf} with nonlinearity $f^L$ given by~\eqref{f l}.  Proposition \ref{rappel FR} implies that $c_H > c_{KPP} = 2\sqrt{d}$, and then, in view of Theorem \ref{effect road}, we can choose $L$ sufficiently large to have
$c_{\star} > c_{KPP}$.

Now, taking $\psi=e^{-\frac c{2d}x}$ in the formula \eqref{fgpe no road}
and using $m^L \leq 1$,
shows that $\lambda_N(c) \geq~\frac{c^2}{4d} - 1$. 
It follows that $\lambda_N(c) \geq0$ when $c\geq c_{KPP}$. From Propositions~\ref{rappel CC} and \ref{persistence no road}, we infer
that $c_N \leq c_{KPP}$. We eventually conclude that
$$
c_N \leq c_{KPP}<c_{\star}.
$$
That is, the two statements of the corollary hold 
with $c_1 := c_N$ and $c_2 := c_\star$.
\end{proof}

\section{Extension to more general reaction terms}\label{unbounded niche}

Throughout the whole paper, up to now, we assumed \eqref{BFZ}, that is, that the favorable zone is bounded. It is natural to wonder whether this condition can be weakened. This question turns out to be rather delicate and it is still by and large open.

Nevertheless, we point out that the notion of generalized principal eigenvalue introduced and studied in \cite{BDR1}, and the basic technical facts recalled in Section \ref{sec rapp vp}, do not require hypothesis \eqref{BFZ}. However, when it comes to studying the long-time behavior and the qualitative properties of road-field models of the type \eqref{mf}, the boundedness of the ecological niche is crucial.

For instance, without \eqref{BFZ}, the uniqueness of stationary solutions is not guaranteed anymore. In addition, it is not clear that the solutions of road-field models \eqref{mf} would converge to stationary solutions as $t$ goes to $+\infty$. As a consequence, the keystone of our analysis, Theorem \ref{persistence} (which states that the sign of $\l$ completely characterizes the long-time behavior of solutions) is not known in this context.

Therefore, in this framework, the notions of persistence and extinction, as stated in Definition \ref{def 1}, do not make much sense. For this reason, we introduce the following modified notions:


\begin{definition}\label{def weak}

	For the system \eqref{mf}, that is, the system in the moving frame, we say that
	
	\begin{enumerate}[(i)]
		\item \emph{local extinction in the moving frame} occurs if every solution arising from a non-negative compactly supported initial datum converges locally uniformly to zero as $t$ goes to $+\infty$;

		\item \emph{local persistence in the moving frame} occurs if every solution arising from a non-negative not identically equal to zero compactly supported initial datum satisfies, for every $R>0$,
		$$
		\liminf_{t\to+\infty} \left(\inf_{x\in I_R} u(t,x)\right) >0, \quad \liminf_{t\to+\infty} \left(\inf_{(x,y)\in \O_R} v(t,x,y)\right) >0.
		$$
	\end{enumerate}
\end{definition}
The fact that we are working in the moving frame is important. Indeed, in a system exhibiting a climate change ($c\neq 0$), if the favorable zone is bounded, i.e., \eqref{BFZ} is verified, then every solution of the road-field model \eqref{frwcc} in the original frame goes to zero locally uniformly as $t$ goes to $+\infty$ (even when we have persistence in the moving frame).

We have the following weak version of Theorem \ref{persistence}:
\begin{theorem}\label{th weak}
	Let $\lambda_{1}$ be the generalized principal eigenvalue of system \eqref{lsmf}, with $m$ not necessarily satisfying \eqref{BFZ}. 
	
	\begin{enumerate}[(i)]
		
		\item If $\lambda_{1}<0$, local persistence in the moving frame occurs for \eqref{mf}.
					
		\item If $\l >0$, local extinction in the moving frame occurs for \eqref{mf}.
		

	\end{enumerate}
\end{theorem}
The proof of this result is similar to that of Theorem \eqref{persistence}. We briefly sketch it here for completeness.
\begin{proof}
	Let $(u,v)$ be a solution of \eqref{mf} arising from a compactly supported non-null initial datum.
	
	\medskip
	\emph{Statement $(i)$.}\\
	As in the proof of statement $(i)$ of Theorem \ref{persistence}, we can take $R>0$ and $\e>0$ such that $\e (\phi_R,\psi_R)$ is a stationary subsolution of \eqref{mf} and such that $(u(1,\cdot),v(1,\cdot,\cdot) \geq \e(\phi_R,\psi_R)$. Then, the parabolic comparison principle yields that the solution of \eqref{mf} arising from the initial datum $\e(\phi_R,\psi_R)$ is time-increasing and converges to a positive stationary solution, hence statement $(i)$ of Theorem \ref{th weak} follows. 
	
	In the proof of statement $(i)$ of Theorem \ref{persistence}, we also used $(\frac{\nu}{\mu}M,M)$, with $M>0$ large enough, as an initial datum in \eqref{mf} to bound from above $(u,v)$. We could do the same here. However, for want of an uniqueness result, the solution of \eqref{mf} arising from $(\frac{\nu}{\mu}M,M)$ may converge to a stationary solution different from the one obtained by starting with the initial datum $\e(\phi_R,\psi_R)$. 
	
	\medskip
	\emph{Statement $(ii)$.}\\
  Observe that, for $M>0$, the pair $Me^{-\l t}(\phi,\psi)$ is supersolution of \eqref{mf}. Taking $A$ large enough so that $(u_0,v_0)\leq A$, the parabolic comparison principle yields that $(u,v)$ goes to zero, locally uniformly as $t$ goes to $+\infty$. The converge to zero is only locally uniform because the generalized principal eigenvalue $(\phi,\psi)$ may not be bounded.
\end{proof} 
%
%
%
Is is natural to wonder whether there are situations where weak extinction occurs but not extinction. In such situation, the population would move inside favorable zones but would neither settle definitively there nor go extinct. We leave this as an open question.




\section{Conclusion}

We have introduced a model that aims at describing the effect of a line with fast
diffusion (a road) on the dynamics of an ecological niche. We incorporate in the model
the possibility of a climate change.
We have found that this model exhibits two contrasting  influences of the road. The first one is that the 
presence of the line with fast diffusion can lead to the extinction of a
population that would otherwise persist: the effect of the line is deleterious. On the other hand, if 
the ecological niche is moving, because of a climate change, then there are situations where 
a population that would otherwise be doomed to extinction manages to survive thanks to the 
presence of the~road.

The first  result is not a priori intuitive: in our model, the line with fast diffusion is not lethal, in the 
sense that there is no death term there. The second result, that is,
the fact that the line can ``help" the population, is also surprising, because there is no reproduction on the line either. 

These results are derived through a careful analysis of the properties of
a notion of {\em generalized principal eigenvalue}
for elliptic systems set in different spatial dimensions, introduced in
our previous work~\cite{BDR1}.\\


\bigskip
\noindent \textbf{Acknowledgements:}  This work has been supported by the ERC Advanced Grant 2013 n. 321186 ``ReaDi -- Reaction-Diffusion Equations, Propagation and Modelling'' held by H.~Berestycki. This work was also partially supported by the French National Research Agency (ANR), within  project NONLOCAL ANR-14-CE25-0013.
This paper was completed while H. Berestycki was visiting the HKUST Jockey Club Institute of Advanced Study of Hong Kong University of Science and Technology, and the Department of Mathematics at Stanford University, and their support is gratefully acknowledged.


\begin{thebibliography}{99}
%
\bibitem{AW}
D.~G. Aronson and H.~F. Weinberger.
\newblock Multidimensional nonlinear diffusion arising in population genetics.
\newblock {\em Adv. in Math.}, 30(1):33--76, 1978.
%
\bibitem{BCRR_sem}
H.~Berestycki, A.-C. Coulon, J.-M. Roquejoffre, and L.~Rossi.
\newblock Speed-up of reaction-diffusion fronts by a line of fast diffusion.
\newblock In {\em S\'eminaire {L}aurent {S}chwartz---\'Equations aux
  {D}\'eriv\'ees {P}artielles et {A}pplications. {A}nn\'ee 2013--2014}, pages
  Exp. No. XIX, 25. Ed. \'Ec. Polytech., Palaiseau, 2014.
%
\bibitem{BCRR}
 H.~Berestycki, A.-C. Coulon, J.-M. Roquejoffre and L.~Rossi.
\newblock The effect of a line with nonlocal diffusion on {F}isher-{KPP}
  propagation.
\newblock {\em Math. Models Methods Appl. Sci.}, 25(13):2519--2562, 2015.
%
%
\bibitem{BDNZ}
H.~Berestycki, O.~Diekmann, C.~J. Nagelkerke, and P.~A. Zegeling.
\newblock Can a species keep pace with a shifting climate?
\newblock {\em Bull. Math. Biol.}, 71(2):399--429, 2009.
%
\bibitem{BDR1}
H.~Berestycki, R.~Ducasse and L. Rossi.
\newblock Generalized principal eigenvalues for heterogeneous road-field systems.
newblock {\em Comm. in Contemporary Math.}, 2019, in press. 
%
\bibitem{BNV} 
H.~Berestycki, L.~Nirenberg, and S.~R.~S. Varadhan.
\newblock The principal eigenvalue and maximum principle for second-order
  elliptic operators in general domains.
\newblock {\em Comm. Pure Appl. Math.}, 47(1):47--92, 1994.
%
%
%
\bibitem{BRR1}
 H.~Berestycki, J.-M. Roquejoffre and L.~Rossi.
\newblock The influence of a line with fast diffusion on {F}isher-{KPP}
  propagation.
\newblock {\em J. Math. Biol.}, 66(4-5):743--766, 2013.
%
\bibitem{BRR2}
H.~Berestycki, J.-M. Roquejoffre and L.~Rossi.
\newblock Fisher-{KPP} propagation in the presence of a line: further effects.
\newblock {\em Nonlinearity}, 26(9):2623--2640, 2013.
%
\bibitem{BRR3}
 H.~Berestycki, J.-M. Roquejoffre and L.~Rossi.
\newblock The shape of expansion induced by a line with fast diffusion in
  fisher-kpp equations.
\newblock {\em Commun. Math. Phys.},  343: 207-232, 2016. 
%
%
%
\bibitem{BRforcedspeed}
H.~Berestycki and L.~Rossi.
\newblock Reaction-diffusion equations for population dynamics with forced
  speed. {I}. {T}he case of the whole space.
\newblock {\em Discrete Contin. Dyn. Syst.}, 21(1):41--67, 2008.
%
\bibitem{BR4}
H.~Berestycki and L.~Rossi.
\newblock Generalizations and properties of the principal eigenvalue of
elliptic operators in unbounded domains.
\newblock {\em Comm. Pure Appl. Math.}, 68(6):1014--1065, 2015.
%
%
\bibitem{Diet1}
 L.~Dietrich.
\newblock Existence of traveling waves for a reaction--diffusion system with a
  line with fast diffusion.
\newblock {\em Appl. Math. Res. Express. AMRX}, (2):204--252, 2015.
 %
 \bibitem{Dcurved} 
R.~Ducasse.
\newblock Influence of the geometry on a field-road model: the case of a
  conical field.
\newblock {\em J. Lond. Math. Soc. (2)}, 97(3):441--469, 2018.
%
\bibitem{car}
R. Eritja, J.~R.~B. Palmer D. Roiz, I. Sanpera-Calbet and F. Bartumeus.
\newblock Direct Evidence of Adult Aedes albopictus Dispersal by Car.
\newblock {\em Scientific Reports}, 7, 14399, 2017.
%
\bibitem{F}
R.~A. Fisher.
\newblock The wave of advantage of advantageous genes.
\newblock {\em Ann. Eugenics}, 7:355--369, 1937.
%
%
%
\bibitem{GMZ}
T.~Giletti, L.~Monsaingeon and M.~Zhou.
\newblock A {KPP} road-field system with spatially periodic exchange terms.
\newblock {\em Nonlinear Anal.}, 128:273--302, 2015.
%
\bibitem{tires}
W.~A. Hawley, P. Reiter, R. Copeland, C.~B. Pumpuni and G.~B. Jr Craig.
\newblock Aedes albopictus in North America: probable introduction in used tires from Northern Asia.
\newblock Science (New York, N.Y.). 236. 1114-6, 1987.
%
\bibitem{W1}
H.~Mckenzie, E.~Merrill, R.~Spiteri, and M.~Lewis.
\newblock How linear features alter predator movement and the functional
response.
\newblock {\em Interface focus}, 2:205--16, 04 2012.

%
\bibitem{W2}
T.~Hillen and K.~J. Painter.
\newblock Transport and anisotropic diffusion models for movement in oriented
habitats.
\newblock In {\em Dispersal, individual movement and spatial ecology}, volume
2071 of {\em Lecture Notes in Math.}, pages 177--222. Springer, Heidelberg,
2013.
%
\bibitem{JB}
T. Jung and M. Blaschke.
\newblock Phytophthora root and collar rot of alders in Bavaria: distribution, modes of spread and possible management strategies.
\newblock {\em Plant Pathology}  53, 197-208, 2004.
%
\bibitem{KPP}
 A.~N. Kolmogorov, I.~G. Petrovski\u{\i} and N.~S. Piskunov.
\newblock {\'E}tude de l'{\'e}quation de la diffusion avec croissance de la
  quantit{\'e} de mati{\`e}re et son application {\`a} un probl{\`e}me
  biologique.
\newblock {\em Bull. Univ. Etat. Moscow Ser. Internat. Math. Mec. Sect. A},
  1:1--26, 1937.
%
\bibitem{KR}
M.~G. Kre{\u\i}n and M.~A. Rutman.
\newblock Linear operators leaving invariant a cone in a {B}anach space.
\newblock {\em Amer. Math. Soc. Translation}, 1950(26):128, 1950.
%
 \bibitem{La}
O.~A. Ladyzenskaja, V.~A. Solonnikov, and N.~N. Ural'ceva.
\newblock {\em Linear and quasilinear equations of parabolic type}.
\newblock Translated from the Russian by S. Smith. Translations of Mathematical
  Monographs, Vol. 23. American Mathematical Society, Providence, R.I., 1967.
%
\bibitem{ML}
M.~J. Millner and L.~R. Loaiza.
\newblock Geographic Expansion of the Invasive Mosquito Aedes albopictus across Panama - Implications for Control of Dengue and
Chikungunya Viruses.
\newblock {\em PLoS Negl Trop Dis} 9(1): e0003383, 2015.
%
%
%
\bibitem{P2}
 A.~Pauthier.
\newblock Uniform dynamics for {F}isher-{KPP} propagation driven by a line with
  fast diffusion under a singular limit.
\newblock {\em Nonlinearity}, 28(11):3891--3920, 2015.
%
\bibitem{P1}
 A.~Pauthier.
\newblock The influence of nonlocal exchange terms on {F}isher-{KPP}
  propagation driven by a line with fast diffusion.
\newblock {\em Commun. Math. Sci.}, 14(2):535--570, 2016.

\bibitem{PL}
 A.B. Potapov, M.A.  Lewis.
\newblock
 Climate and competition: the effect of moving range boundaries on
habitat invasibility. 
 \newblock {\em Bull. Math. Biol.} 66, 975--1008, 2004.
%
%

\bibitem{rnhf}
I. Rochlin, D.~V. Ninivaggi, M.L. Hutchinson, A. Farajollahi.
\newblock Climate Change and Range Expansion of the Asian Tiger Mosquito (Aedes albopictus) in Northeastern USA: Implications for Public Health Practitioners. 
\newblock {\em PLOS ONE}, 8(4): e60874, 2013.
%
\bibitem{RTV}
L.~Rossi, A.~Tellini, and E.~Valdinoci.
\newblock The effect on {F}isher-{KPP} propagation in a cylinder with fast
  diffusion on the boundary.
\newblock {\em SIAM J. Math. Anal.}, 49(6):4595--4624, 2017.
%
\bibitem{Si}
A. Siegfried.
\newblock Itin\'eraires des contagions, \'epid\'emies et id\'eologies.
\newblock Paris, A. Colin, 164 pages, 1960.

\end{thebibliography}
\end{document}